\documentclass [11pt,oneside]{amsart}

\reversemarginpar \textwidth=12.5cm \textheight=19.5cm

\usepackage{amssymb} \usepackage{amsfonts} \usepackage{amsmath}
\usepackage{amsthm} 
\usepackage{epsfig} 
\usepackage{tikz}
\usepackage[cmtip,arrow]{xy}
\usepackage{amsmath,amssymb,enumerate,pb-xy}

\usepackage[applemac]{inputenc}

\usepackage{mathrsfs} 
\usepackage{verbatim} 
\usepackage{hyperref}

\usepackage[small,nohug,heads=littlevee]{diagrams} 
\diagramstyle[labelstyle=\scriptstyle]

\input xy
\xyoption{all}

\usepackage{caption}

\newtheorem{dummy}{Dummy}[section]

\newtheorem{prop}[dummy]{Proposition}
\newtheorem{cor}[dummy]{Corollary} 
\newtheorem{theorem}[dummy]{Theorem}
\newtheorem{lem}[dummy]{Lemma}

\theoremstyle{definition}

\theoremstyle{definition}

\newtheorem{dfn}[dummy]{Definition}

\newtheorem{rem}[dummy]{Remark}
\newtheorem{ntz}[dummy]{Notation}

\theoremstyle{remark}

\newcommand{\N}{\mathbb{N}}

\newcommand{\Z}{\mathbb{Z}}
\newcommand{\Q}{\mathbb{Q}}

\newcommand{\K}{\mathbb{K}}

\newcommand{\loc}{\mathrm{{loc}}}
\newcommand{\thin}{\mathrm{{thin}}}

\renewcommand{\mod}{\,\,\mathrm{{mod}}\,\,}
\newcommand{\simp}{\mathrm{{simp}}}
\newcommand{\maxh}{\mathrm{{maxh}}}
\newcommand{\height}{\mathrm{{height}}}
\newcommand{\minh}{\mathrm{{minh}}}

\newcommand{\diam}{\mathrm{{diam}}}
\newcommand{\stab}{\mathrm{{Stab\,}}}

\renewcommand{\max}{\mathrm{{max}}}

\newcommand{\morf}{\hom}

\newcommand{\rel}{\mathrm{rel}}

\newcommand{\St}{\,\mathrm{St}\,}
\newcommand{\st}{\,\mathrm{St}\,}

\renewcommand{\:}{\mathrm{\,:\,}}

\newcommand{\obj}{\mathrm{\,obj\, }}
\renewcommand{\hom}{\mathrm{\,Hom\,}}

\newcommand{\id}{\mathrm{\,Id\,}}
\newcommand{\supp}{\mathrm{\,Supp\,}}

\newcommand{\R} {\ensuremath {\mathbb{R}}}

\author{Federico Franceschini}

\address{Dipartimento di Matematica \\
Universit\`a di Pisa \\
Largo B.~Pontecorvo 5 \\
56127 Pisa, Italy}

\email{franceschini@mail.dm.unipi.it}

\title[Bounded cohomology and relative hyperbolicity]{A characterization of relatively hyperbolic groups via bounded cohomology}



\thanks{}

\begin{document}

\begin{abstract} It was proved by Mineyev and Yaman that, if $(\Gamma, \Gamma')$ is a relatively hyperbolic pair, the comparison map 
$$ H_b^k(\Gamma, \Gamma'; V) \to H^k(\Gamma, \Gamma'; V) $$
is surjective for every $k \ge 2$, and any bounded $\Gamma$--module $V$. By exploiting results of Groves and Manning, we give another proof of this result. Moreover, we prove the opposite implication under weaker hypotheses than the ones required by Mineyev and Yaman.
\end{abstract}

\maketitle
%$$ \today $$
%\tableofcontents

\section{Introduction}

The relations between bounded cohomology and geometric group theory have been proved to be fruitful on several occasions. For instance, the second bounded cohomology with real coefficients of most hyperbolic groups has uncountable dimension (\cite{EpsteinFujiwara1997}). This results generalizes an analogous fact for free non-abelian groups (see \cite{Brooks}, or \cite{Rolli} for a simpler proof) and was in turn extended by considering groups acting properly discontinuously on Gromov hyperbolic spaces (\cite{Fujiwara1998}). The proper discontinuity condition was weakened in order to include other interesting classes of group actions on Gromov hyperbolic spaces where a Brook's type argument could be applied. For example, the WPD (weakly properly discontinuous) property and the acylindrical hyperbolicity were introduced by Bestvina and Fujiwara (\cite{BestFuji2002}) and Bowditch (\cite{Bow2003}) respectively in order to study actions of mapping class groups on curve complexes. The second bounded cohomology for more complicated coefficients of (most) acylindrically hyperbolic groups was shown to be infinite-dimensional in \cite{HullOsin2013} and \cite{BestBrombFuji}.

Two other cases somehow opposite to each other are the characterization of amenability in terms of the vanishing of bounded cohomology (\cite{Johnson}) and the characterization of Gromov hyperbolicity of groups in term of the surjectivity of the comparison map in higher degrees (\cite{Min2}). The last two examples could be exploited to prove that the simplicial volume of connected closed (and aspherical of dimension at least 2) oriented manifolds with amenable (Gromov hyperbolic) fundamental group vanishes (is nonzero).

In the present paper we will consider a generalization of Mineyev's result to the relative setting. The absolute case was considered by Mineyev in \cite{Min1} and \cite{Min2}. He proved that, if $\Gamma$ is hyperbolic, the comparison map $ H_b^k(\Gamma, V) \to H^k(\Gamma, V) $ is surjective for every $k \ge 2$ and every bounded $\Gamma$--module $V$. Viceversa, if $\Gamma$ is finitely presented and the comparison map $ H_b^2(\Gamma, V) \to H^2(\Gamma, V) $ is surjective for every bounded $\Gamma$--module $V$, then $\Gamma$ is hyperbolic (actually, it was proven by Gromov and Rips that hyperbolic groups are finitely presented: see \cite[Corollary 2.2.A]{GromovHG} or \cite[Théorème 2.2]{CDP}).

In this work we consider a relative version of the results of \cite{Min1} and \cite{Min2} which holds for group-pairs, i.e. couples $(\Gamma, \Gamma')$ where $\Gamma$ is a group and $\Gamma'$ is a finite family of subgroups of $\Gamma$. The following is our main result (see Section \ref{sezione 2} for the definitions of the terms involved).
\begin{theorem} \label{main MIO theorem} Let $(\Gamma, \Gamma')$ be a group-pair. \begin{enumerate}
    \item[(a)] If $(\Gamma, \Gamma')$ is relatively hyperbolic the comparison map
$$ H_b^k(\Gamma, \Gamma'; V) \to H^k(\Gamma, \Gamma'; V) $$
is surjective for every bounded $\Gamma$--module $V$ and $k \ge 2$.
    \item[(b)] Conversely, if $(\Gamma, \Gamma')$ is a finitely presented group-pair such that $\Gamma$ is finitely generated and the comparison map is surjective in degree $2$ for any bounded $\Gamma$--module $V$, then $(\Gamma, \Gamma')$ is relatively hyperbolic. 
\end{enumerate}
\end{theorem}

Roughly speaking, the group-pair $(\Gamma, \Gamma')$ is finitely presented if there is a presentation for $\Gamma$ in the alphabet $\bigsqcup_{i \in I} \Gamma_i \sqcup \mathscr A$ -- where $ \mathscr A \subset \Gamma$ is a finite set -- such that only finitely many relations involve elements of $ \mathscr A$. See Definition \ref{q3o47bfihuadjks} for more details. 

Mineyev and Yaman proved in \cite{MY} a similar theorem. In particular, they proved (a), while the opposite implication was proved only under stronger hypotheses than (b) above (see \cite[Theorem 59]{MY}).

In an article of Groves and Manning (\cite{GM}) written shortly thereafter, several useful results are proved which seem to provide an alternative strategy to prove (a) of Theorem \ref{main MIO theorem}. Indeed, quoting from \cite[p. 4]{GM}: \begin{quote} `` In particular, in \cite{GM2}, we define a homological bicombing on the coned-off Cayley graph of a relatively relatively hyperbolic group (using the bicombing from this paper in an essential way) in order to investigate relative bounded cohomology and relatively hyperbolic groups, in analogy with \cite{Min1} and \cite{Min2}. ''
\end{quote}
The article \cite{GM2} was referred to as ``in preparation'', and has never appeared. It was our aim to provide such a proof. We take a small detour from the strategy outlined in the quotation above, since we will use the \emph{cusped-graph} defined in \cite{GM} instead of the coned-off Cayley graph.

In \cite{MY} a weaker version of (b) is also considered. However, such implication was proved under additional \emph{finiteness} hypotheses about the action of $\Gamma$ on a graph or complex, which seem to be far more restricting than the finite presentability in the absolute case. By making use of recent results in a paper of Pedroza \cite{pedro}, we will be able to prove this implication with a proof similar to the one in \cite{MY}, but without mentioning $\Gamma$--actions in the statement.

In Section \ref{sezione 7} we give two applications. The first one is a straightforward consequence of Theorem \ref{main MIO theorem} (a) and was already proved in \cite{MY}: if the topological pair $(X, Y)$ is a classifying-pair for $(\Gamma, \Gamma')$, then the Gromov norm on $H_k(X, Y)$ -- which in general is merely a semi-norm -- is actually a norm, for $k \ge 2$. This implies in particular interesting non-vanishing results for some classes of compact manifolds with boundary. The second application easily follows from our Rips complex construction, and can be obtained in the same way from an analogous construction in \cite[Section 2.9]{MY}. It states that, for a hyperbolic pair $(\Gamma, \Gamma')$, there is $n \in \N$ such that, for any $\Gamma$--module $V$, the relative (non-bounded) cohomology of $(\Gamma, \Gamma')$ with coefficients in $V$ vanishes in dimensions at least $n$.

The plan of the paper is as follows. In sections \ref{sezione 2} and \ref{sezione 3} we recall some definitions and results from \cite{MY} and \cite{GM} (some technicalities pertaining to Section \ref{sezione 2} are addressed later in the first addendum). In sections \ref{sezione 3}, \ref{sezione 4} and \ref{sezione 5} we introduce a Rips complex construction as our main tool, and prove some filling-inequalities of its simplicial chain complex, which will allow us to prove Theorem \ref{main MIO theorem} (a) in Section \ref{sezione 6}. In the following section we give the applications already mentioned. In Section \ref{sezione 8} we recall some results in \cite{pedro} and prove Theorem \ref{main MIO theorem} (b). In the second addendum we show that the definitions of relative bounded cohomology given in \cite{MY} and \cite{Blank} respectively are isometric.

\subsection*{Acknowledgements} This work is part of a Ph.D. project that I am developing under the supervision of Roberto Frigerio. I would like to thank him for several conversations on this problem, and for having carefully read previous versions of the present work. I am also grateful to Matthias Blank and Clara L\"oh for a pleasant time I spent in Regensburg, where some of the contents of this article were discussed.

\section{\label{sezione 2}Preliminaries}

Several definitions and results in this section are taken from \cite{MY}. 

\label{g3fehrclns} Given a set $S$, let $\R S$ be the vector space with basis $S$. Then $S$ induces a natural $\ell^1$--norm $\|\cdot\|$ on $\R S$
$$ \Big\|\sum_{s \in S} \lambda_s s \Big\| := \sum_{s \in S} |\lambda_s| $$
(where almost all coefficients $\lambda_s$ are null). We denote by $C_*(S)$ the complex defined by
$$ C_{k}(S) = \left\{0\right\} \quad \textrm{if } k \le -1 \qquad C_{k}(S) = \R S^{k+1} \quad \textrm{if } k \ge 0, $$
with boundary operator given by
$$ \partial_k(s_0, \ldots, s_k) := \sum_{j=0}^k (-1)^k (s_0, \ldots, \hat s_j, \ldots, s_k). $$
Notice that $ \partial_k$ is a bounded linear operator for every $k$. If $\Gamma$ is a group acting on $S$, then $\Gamma$ also acts diagonally on $C_k(S)$ via isometries, and $\partial_k$ is $\Gamma$--equivariant with respect to this action. The complex $C_*(S)$ admits an exact augmentation given by
$$ C_0(S) \to \R \qquad \sum_i \lambda_i s_i \mapsto \sum_i \lambda_i. $$ 

The following definition of relative bounded cohomology is taken from \cite{MY} and is modelled on the analogous one for the non-bounded version in \cite{BE}. Our notation is slightly different from that of \cite{MY}.

\begin{dfn} \label{ } A $\Gamma$--module is a real vector space equipped with a linear $\Gamma$--action. A $\Gamma$--module $P$ is \textbf{projective} if, given $\Gamma$--equivariant maps $\varphi \: V \to W$ and $f \: P \to W$, with $\varphi$ surjective, there exists a $\Gamma$--equivariant map $\tilde f \: P \to V$ making the following diagram commute
\begin{equation} \label{fehiruj}
\begin{diagram} 
&& V\\
& \ruTo^{\tilde f} & \phantom{\,\, \varphi} \dTo \,\, \scriptstyle{\varphi}\\
P & \overset{f}{ \longrightarrow} & W. \\
\end{diagram}
\end{equation}

Given a module $M$, a \textbf{$\Gamma$--resolution} for $M$ is an exact $\Gamma$--complex
$$ \cdots E_k \to \cdots \to E_0 \to M \to 0. $$
A \textbf{$\Gamma$--projective resolution of $M$} is a $\Gamma$--resolution where all the $E_i$ are $\Gamma$--projective.
\end{dfn}

The following lemma, (similar to \cite[Lemma 52]{MY}) will be useful.

{\begin{lem} \label{Lemma 52 MY} Let $P$ be a $\Gamma$--module generated as a vector space by a basis $S$. Suppose that the action of $\Gamma$ on $P$ is such that, for every $s \in S$ and $\gamma \in \Gamma$, there is $t \in S$ such that $\gamma s = \pm t$. Moreover, suppose that $|\stab_\Gamma (s)| < \infty$ for every $s \in S$. Then $P$ is a $\Gamma$--projective module.
\end{lem}

\begin{proof} Let $\varphi \: V \to W$ and $f \: P \to W$ be $\Gamma$--equivariant maps, and suppose that $\varphi$ is surjective. If $a \in P$, let $\stab^- (a) := \{\gamma \in \Gamma \: \gamma a = - a\}$. Notice that $|\stab^- (s)|$ is null or equals $|\stab (s)|$, hence in particular it is finite, if $s \in S$. Fix $s \in S$ and $b \in V$ such that $f(s) = \varphi(b)$. Put
$$ \tilde f(\pm \alpha s) := \pm \frac{\sum_{\gamma \in \stab (s)} \gamma \alpha b - \sum_{\gamma \in \stab^- (s)} \gamma \alpha b}{\left| \stab(s) \cup \stab^-(s) \right|} \qquad \forall \alpha \in \Gamma $$
The definition above gives rise to a well defined $\R$--linear and $\Gamma$--equivariant map $\R\Gamma s \to V$. Since $P$ is a direct sum of spaces of type $\R \Gamma$, $s \in S$, we obtain a $\Gamma$--equivariant map $ \tilde f \: P \to V$. Finally, it is easy to see that $\varphi \circ \tilde f = f$.
\end{proof}
In particular, if $\Gamma$ acts freely on $S$, then $C_*(S) \to \R \to 0$ is a $\Gamma$--projective resolution of the trivial $\Gamma$--module $\R$. 

We also have a normed version of projectivity.
\begin{dfn} \label{ } Let $\Gamma$ be a discrete group. A \textbf{bounded $\Gamma$-module} $V$ is an $\R$--normed space equipped with a (left) $\Gamma$-action of equibounded automorphisms, i.e. there exists $ L > 0$ such that
$$ \|\gamma\cdot v\| \le L \|v\| \qquad \forall v \in V,\, \gamma \in \Gamma. $$
A \textbf{bounded $\Gamma$--complex} is a complex of bounded $\Gamma$--modules with $\Gamma$--equivariant bounded boundary operators.
\end{dfn}

{
\begin{dfn} \label{348bfy7oewdilnu} A map $\varphi \: V \to W$ between normed spaces is \textbf{undistorted} if there exists $K > 0$ such that, for every $w \in W$ in the image of $\varphi$, there exists $v \in V$ such that
$$ \varphi(v) = w, \qquad \|v\| \le K \|w\|. $$ 
\end{dfn}

\begin{dfn} \label{ } A $\Gamma$--module $P$ is \textbf{$b$--projective} if, given any surjective undistorted bounded $\Gamma$--map $\varphi \: V \to W$ and any bounded $\Gamma$--map $f \: P \to W$, there exists a bounded $\Gamma$--map $\tilde f \: P \to V$ making the following diagram commute
\begin{equation} \label{fehiruj}
\begin{diagram}
&& V\\
& \ruTo^{\tilde f} & \phantom{\,\, \varphi} \dTo \,\, \scriptstyle{\varphi}\\
P & \overset{f}{ \longrightarrow} & W. \\
\end{diagram}
\end{equation}

Given a module $M$, a \textbf{bounded $\Gamma$--resolution} for $M$ is an exact bounded $\Gamma$--complex
$$ \cdots E_k \to \cdots \to E_0 \to M \to 0. $$
A \textbf{$b$--projective resolution of $M$} is a bounded $\Gamma$--resolution of $M$ where all the $E_i$ are $b$--projective and all maps are undistorted.
\end{dfn}

}

Given (bounded) $\Gamma$--modules $V$ and $W$, we denote by $\hom_{(b)}(V, W)$ the space of all (bounded) $\R$--linear homomorphisms from $V$ to $W$, and we denote by $\hom_{(b)}^\Gamma(V, W)$ the subspace of $\hom_{(b)}(V, W)$ whose elements are $\Gamma$--equivariant.

The following lemma is a simple exercise in homological algebra: 
\begin{lem} \label{q3onvp31t4y2} Given two ($b$--)projective $\Gamma$--resolutions $E_M$ and $E_M'$ of the same module $M$, there exists a (bounded) chain $\Gamma$--map $\varphi_* \: E_M \to E_M'$ which extends the identity on $M$. This map is unique up to (bounded) $\Gamma$--homotopy. 

Dually, if $V$ is any (bounded) $\Gamma$--module and $\varphi_1$, $\varphi_2 \: E_M \to E_M'$ are as above, there is a (bounded) $\Gamma$--homotopy between $\varphi_1^*$ and $\varphi_2^* \: \hom_{(b)}^*(E_M', V) \to \hom_{(b)}^*(E_M, V)$.
\end{lem}

Notice that, for every $\Gamma$--set $S$, the space $\R S$ is a bounded $\Gamma$--module and $C_*(S)$ is a bounded $\Gamma$--complex, whose augmentation is a $\Gamma$--projective and $\Gamma$--$b$--projective resolution of $\R$, if $\Gamma$ acts on $S$ as in Lemma \ref{Lemma 52 MY}.

\begin{dfn} \label{ } Let $\Gamma$ be a group, and let $\Gamma' := \{\Gamma_i\}_{i \in I}$ be a finite non-empty parametrized family of subgroups (this means that we allow repetitions among the $\Gamma_i$). We call such $(\Gamma, \Gamma')$ a \textbf{group-pair}. 
\end{dfn}

\begin{dfn} \label{standard resolutions} Given a group-pair $(\Gamma, \Gamma')$, let $I\Gamma$ be the $\Gamma$--set $\bigsqcup_{i \in I} \Gamma \sim \Gamma \times I$ (where $\Gamma$ acts on $I\Gamma$ by left translation of each copy of $\Gamma$). We consider the complex
$$ \St = \St_*(I\Gamma) := C_*(\Gamma \times I). $$
Let $\St'$ be the $\Gamma$-subcomplex of $\St$ with basis given by the tuples $(x_0, \ldots, x_k) \in (\Gamma \times I)^{k+1}$ for which there exists $i \in I$ such that $x_j \in \Gamma \times \{i\}$ for all $0 \le j \le k$ and $x_j \in x_0\Gamma_i$ for every $1 \le j \le k$. Finally, let $\St_*^{\rel} := \St_*/\St_*'$ be the quotient $\Gamma$--complex. If $V$ is a (bounded) $\Gamma$--module, the \textbf{(bounded) cohomology of the group-pair $(\Gamma, \Gamma')$ with coefficients in $V$} is the cohomology of the cocomplex
$$ \St_{(b)}^{\rel \, *}(\Gamma, \Gamma'; V) := \hom_{(b)}^\Gamma(\St_*^{\rel}, V), $$
and it is denoted by $H_{(b)}^*(\Gamma, \Gamma'; V)$.
\end{dfn}

The complex $\St_*^{\rel}(\Gamma, \Gamma')$ is provided with a natural norm, hence we can equip $\St_{(b)}^{\rel \, *}(\Gamma, \Gamma'; V)$ with the corresponding $\ell^\infty$ norm, which descends to a semi-norm on $H_{(b)}^*(\Gamma, \Gamma'; V)$. 

By Lemma \ref{Lemma 52 MY} it is easily seen that $\St_*^{\rel}(\Gamma, \Gamma')$ induces a $\Gamma$--projective resolution of the $\Gamma$--module $\Delta := \ker \,(\R (\Gamma / \Gamma') \to \R)$. Moreover, Lemma \ref{Lemma 52 MY} could be easily adapted to the normed setting, proving that $\St_k^{\rel}(\Gamma, \Gamma')$ is $b$--projective for all $k \ge 2$. Mineyev and Yaman also proved that the boundaries of the complex $\St_*^{\rel} \to \Delta \to 0$ are undistorted, hence the resolution $St_*^\rel$ is $b$--projective (see \cite[Section 8.3]{MY}). It follows by Lemma \ref{q3onvp31t4y2} that the relative (bounded) cohomology of $(\Gamma, \Gamma')$ is computed by any $\Gamma$--equivariant ($b$-)projective resolution of $\Delta$ up to canonical (bilipschitz) isomorphism. Even if we don't actually use the fact that $\St_*^{\rel}$ provides a $b$--$\Gamma$--projective resolution of $\Delta$, we will use the following result (proven in \cite[Section 10]{MY}). For completeness we provide a proof of it in Addendum \ref{addendum i}.

\begin{prop} \cite[The relative cone]{MY} \label{The relative cone} Fix $y \in I\Gamma$. There is a (non-$\R$--linear) map: 
$$ [y, \,\cdot]_{\rel} \: \St_1^{\rel} \to \St_2^{\rel} $$
called the \emph{\textbf{relative cone}}, such that $ \|[y, b]_{\rel}\| \le 3 \|b\|$ for all $b \in \St_1^{\rel}$ and $ \partial [y, z] = z $ for any cycle $z \in \St_1^{\rel}$ with respect to the augmentation map: $\St_1^{\rel} \to \Delta$.
\end{prop}

It follows in particular that
\begin{cor} \cite[Equation (29), p. 38]{MY} \label{Equation (29), p. 38} Fix $y \in I\Gamma$. Let $\beta \in \St_2^\rel(\Gamma, \Gamma; V)$. Then: $\beta - [y, \partial \beta]_{\rel} \in \St_2^\rel(\Gamma, \Gamma; V)$ is a cycle, and therefore also a boundary by the exactness of $\St_*^{\rel}$. Hence, if $\alpha \in \St_b^2(\Gamma, \Gamma'; V)$ is a cocycle, we have
\begin{equation} \label{hlrueij3gkyu} \left\langle \alpha, \beta \right\rangle = \left\langle \alpha, [y, \partial \beta]_{\rel} \right\rangle \end{equation}
\end{cor}

\begin{rem} \label{} A more general notion of relative bounded cohomology for pairs of \emph{groupoids} is developed in \cite{Blank}. By unravelling the definition of relative bounded cohomology given in \cite[Definition 3.5.1 and 3.5.12]{Blank}, it is possible to see that those definitions are isometrically isomorphic. We refer the reader to Proposition \ref{are naturally isometric} in Addendum \ref{addendum ii} for a proof of this fact. 
\end{rem}

\section{\label{sezione 3}Hyperbolic group-pairs and cusped-graph construction} 

Given a graph $G$, we denote by $d := d_G$ the \emph{graph-metric} on $G$. This is the path-metric on $G$ induced by giving length $1$ to every edge in $G$. Now, let $Y$ be a simplicial complex, with $1$--skeleton $Y^{(1)} = G$. Given a vertex $v_0 \in Y^{(0)}$ and a number $R \ge 0$, we define the ball $B_R(v_0)$ with radius $R$ centered in $v_0$ as the full subgraph of $Y$ whose vertex set is $\{v \in Y^{(0)} = G^{(0)} \: d_G(v, v_0) \le R\}$. Notice that this definition is slightly in contrast with the usual notion of balls in metric spaces, since we do not equip the whole $Y$ with a metric if $\dim Y \ge 2$ and, even if $Y = G$, there could be a point $p$ in the middle of an edge $e$ such that $p \in B_R(v_0)$, but $d_G(p, v_0) > R$. More generally, if $A \subseteq Y^{(0)}$ and $r \in \N$, we denote by $\mathscr N_r(A)$ the full subcomplex of $Y$ whose vertex set is $ \left\{v \in Y^{(0)}\: d_G(v, A) \le r\right\}$.

Let $S \not\ni 1$ be a symmetric finite generating set of a group $\Gamma$, and consider the associated \textbf{simplicial Cayley graph} $G^\simp(\Gamma, S)$. This is the simplicial graph (i.e. no double edges allowed) whose vertex set is $\Gamma$, and with a single edge connecting $\gamma_1$ with $\gamma_2$ in $\Gamma$ if and only if $\gamma_1 \gamma_2^{-1} \in S$. In Section \ref{sezione 8} we will consider a non-simplicial version of that graph. 

There are many equivalent definitions of relative hyperbolicity for a group-pair $(\Gamma, \Gamma')$. We choose the one introduced in \cite[p. 21, Definition 3.12; p. 25, Theorem 3.25(5)]{GM} which is based on the following \emph{cusped-graph} construction. In particular, we will restrict our attention to the case when $\Gamma$ is finitely generated and $\Gamma'$ is a finite family of finitely generated subgroups of $\Gamma$. A \textbf{(combinatorial) horoball} $ \mathscr H = \mathscr H(G)$ on a graph $G$ is the graph whose vertex set is parametrized by $G^{0} \times \N$, and with the following edges: \begin{itemize}
    \item the full subgraph of $ \mathscr H$ whose vertex set is $G^{(0)} \times \left\{0\right\}$ is a copy of $G$;
    \item there is a single edge between $(g, n)$ and $(g, n+1)$, for every $(g, n) \in G \times \N$;
    \item there is a single edge between $(g, n)$ and $(h, n)$ if and only $d_G(g, h) \le 2^n$.
\end{itemize}

\begin{dfn} [Cusped-graph] \label{definizione di cusped graf} Let $(\Gamma, \Gamma' = \{\Gamma_i\}_{i \in I})$ be a group-pair of finitely generated groups, and consider a symmetric finite generating set $S \not\ni 1$ of $\Gamma$ such that $S \cap \Gamma_i$ is a finite generating set of $\Gamma_i$ for every $i$ (i.e. $S$ is \textbf{compatible}). For every $i \in I$ and left coset $g \Gamma_i$ of $\Gamma_i$ in $\Gamma$ we consider the combinatorial horoball on the subgraph $g \, G^\simp(\Gamma_i, S \cap \Gamma_i)$ of $G^\simp(\Gamma, S)$. We glue those horoballs to $G^\simp(\Gamma, S)$ in the obvious way (see \cite[p. 18]{GM} for more details). We obtain in this way the \textbf{cusped-graph $X$}.
\end{dfn}

We denote by the triple $(g, i, n) \in \Gamma \times I \times \N$ a vertex of the cusped-graph. Notice that $(g, i, 0)$ and $(g, j, 0)$ denote the same vertex for all $i$, $j \in I$. We call the parameter $n$ in $(g, i, n)$ the \textbf{height} of the vertex $(g, i, n)$. Given a natural number $n$ and a horoball $ \mathscr H$, the $n$--horoball associated with $ \mathscr H$ is the full subgraph $ \mathscr H_n$ of $X$ whose vertices are the ones contained in $ \mathscr H$ with height at least $n$.

We will need the following result from \cite{GM}.

\begin{prop} \cite[Lemma 3.26]{GM} \label{Lemma 3.26 GM} If the cusped-graph $X$ constructed in Definition \ref{definizione di cusped graf} is $\delta$--hyperbolic and $C > \delta$, then the $C$--horoballs are convex in $X$.
\end{prop}

\begin{rem} \label{q0378yoehilrujwd} From now on we fix some constant $C > \delta$, $C \ge 1$.
\end{rem}

\begin{rem} \label{3iyewhilnj} Notice that, by our definition, a cusped-graph is necessarily simplicial. Groves and Manning explicitely allow multiple edges in their definition of cusped-graph. We avoid double edges because we want to consider a cusped-graph as contained in every \emph{Rips complex} over it (see the next section). By Remark \ref{i3qrhiwul}, we can apply all relevant results of \cite{GM} also in our setting.
$\hfill \blacksquare$
\end{rem}

\begin{dfn} $($\cite[Definition 3.12; Theorem 3.25(5)]{GM}$)$ \label{ } Let $(\Gamma, \Gamma')$ be a group-pair of finitely generated groups. The pair $(\Gamma, \Gamma')$ is \textbf{(relatively) hyperbolic} if the cusped-graph of $(\Gamma, \Gamma')$ is a Gromov hyperbolic metric space (with the graph metric).\end{dfn}

\section{\label{sezione 4}Rips complexes on cusped graphs}

\begin{dfn} \label{ } Given a graph $G$ and a parameter $1 \le \kappa \in \N$, the \textbf{Rips complex $ \mathscr R_\kappa(G)$ on $G$} is the simplicial complex with the same $0$--skeleton as $G$, and an $n$--dimensional simplex for every set of $n+1$ vertices whose diameter (with respect to the metric of $G$) is at most $\kappa$.
\end{dfn}

Notice that, since $k \ge 1$, $G$ is naturally a subcomplex of $ \mathscr R_\kappa(G)$. We need the following fundamental result about Rips complexes over Gromov hyperbolic graphs.

\begin{lem} \label{pag. 469 proposizione 3.23 Bridson} Let $G$ be a $\delta$--hyperbolic graph. Then $ \mathscr R_\kappa(G)$ is contractible for every $\kappa \ge 4 \delta + 6$.
\end{lem}

By considering the proof of Lemma \ref{pag. 469 proposizione 3.23 Bridson} given in \cite[Proposition 3.23]{BriHae}, it is possible to derive a more precise version of this lemma (see Corollary \ref{ohq3iulrejk}).

\begin{ntz} \label{q3ohitufwq3liy4gt} Let $G$ be a graph, and let $ \mathscr R = \mathscr R_\kappa(G)$ be a Rips complex over $G$. Then $G$ and $ \mathscr R$ induce two metrics $d_G$ and $d_{ \mathscr R}$ on $G^{(0)} = {\mathscr R}^{(0)}$. For $R \ge 0$ and a vertex $v_0$, we denote the full subcomplex of $ \mathscr R$ whose vertex set is $\{x \in G^{(0)} \: d_G(x, x_0) \le R\}$ by $B_R^G(v_0)$, and refer to it as a $G$--ball.
\end{ntz}

Given a Rips complex $\mathscr R_\kappa(G)$ over $G$, we have, for every $l \in \N$ and every vertex $v$, the equality
\begin{equation} \label{hdijlwo5hgur} B_{l\kappa}^G(v) = B_l(v). \end{equation}

\begin{dfn} \label{ } Given a topological space $Z$ and two subspaces $W_1$ and $W_2$, we say that there is a homotopy from $W_1$ to $W_2$ if the inclusion $W_1 \hookrightarrow Z$ is homotopic to a map $f \: W_1 \to Z$ whose image is $W_2$.
\end{dfn}

A (geometric) simplex in a simplicial complex $Z$ is determined by the set of its vertices. If $x_0, \ldots, x_n$ are non-necessarily distinct vertices in $Z$, we denote by $[x_0, \ldots, x_n]$ the corresponding simplex (if there is one). Notice that the dimension of $[x_0, \ldots, x_n]$ could be less than $n$.

\begin{dfn} \label{45op7qhuefj} If $Z$ is a simplicial complex, $W_1$ and $W_2$ are subcomplexes of $Z$, and $w_1 \in W_1$ and $w_2 \in W_2$ are vertices, we say that $W_2$ is obtained from $W_1$ by \textbf{pushing} $w_1$ toward $w_2$ if the following conditions hold: \begin{enumerate}
    \item for every set of vertices $\{x_0, \ldots, x_n\} \in Z^{(0)} \setminus \left\{w_1\right\}$, $[x_0, \ldots, x_n, w_1]$ is a simplex in $W_1$ if and only if $[x_0, \ldots, x_n, w_2]$ is a simplex in $W_2$;
    \item in that case, $[x_0, \ldots, x_n, w_1, w_2]$ is a simplex in $Z$.
\end{enumerate}
\end{dfn}

Notice that it follows that $W_2^{(0)} = \left(W_1^{(0)} \setminus \left\{w_1\right\}\right) \cup \left\{w_2\right\}$.
Under the conditions of Definition \ref{45op7qhuefj}, there is an obvious simplicial homotopy from $W_1$ to $W_2$.

\begin{lem}\cite[Proposition 3.23]{BriHae} \label{ofehwuil} Let $G$ and $ \mathscr R_\kappa = \mathscr R_\kappa(G)$ be as in Lemma \ref{pag. 469 proposizione 3.23 Bridson}. Let $K$ be a compact subcomplex of $ \mathscr R_\kappa$, and let $v_0 \in \mathscr R_\kappa^{(0)}$ be a vertex. Then, it is possible to inductively homotope the complex $K$ into a sequence of subcomplexes $K_0 = K$, $K_1$, \ldots, $K_m = \{v_0\}$ in such a way that: \begin{enumerate}
    \item there is a sequence of vertices $x_i \in K_i^{(0)}$ such that 
    $$d_G(v_0, x_i) = \max\{d_G(v_0, y) \: y \in K_i^{(0)} \}, $$
    \item $K_{i+1}$ is obtained from $K_i$ by pushing $x_i$ toward some vertex $y_{i}$ such that $d_G(v_0, y_i) < d_G(v_0, x_i)$. 
\end{enumerate}

\end{lem}

\begin{cor} \label{3qgrwe3hqru} Let $G$ be a $\delta$--hyperbolic locally compact graph and let $\kappa \ge 4 \delta + 6$. Then every $G$--ball $B_R^G(v_0) \subseteq \mathscr R_\kappa(G)$ is a contractible topological space.
\end{cor}

\begin{proof} In the notations of Lemma \ref{ofehwuil} simply note that, by point (2), the $K_i$ are contained in $B_R^G(v_0)$.
\end{proof}

Given a Rips complex $ \mathscr R_\kappa(X)$ over some cusped space $X$, an ($n$--)horoball of $\mathscr R_\kappa(X)$ is the full subcomplex of $\mathscr R_\kappa(X)$ having the same vertices of an ($n$--)horoball of $X$. Recall that we have fixed a constant $C > \delta$ (Remark \ref{q0378yoehilrujwd}).

\begin{cor} \label{ohq3iulrejk} Let $X$ be the cusped space of a relatively hyperbolic group-pair $(\Gamma, \Gamma')$ (with respect to some finite generating set $S$ as described above) and let $\delta$ be a hyperbolicity constant of $X$, which we can assume to be an integer. Then, for $\kappa \ge 4 \delta + 6$, the Rips complex $\mathscr R = \mathscr R_\kappa(X)$ is contractible, with contractible $C$--horoballs.  Moreover, the balls of $ \mathscr R_\kappa(X)$ are also contractible.
\end{cor}

\begin{proof} The last assertion follows from Corollary \ref{q3ohitufwq3liy4gt} and Equation \eqref{hdijlwo5hgur}. Now, let $K$ be a compact subcomplex contained in some $C$--horoball  $ \mathscr H_C$ (recall that $ \mathscr H_C$ is convex). Let $v_L = (g, i, n)$ be the lowest vertex of $K$, and let $D := \max\{d_X(v_L, v) \: v \in K^{(0)}\}$. Then, it is easy to see that $K$ is contained in the $X$--ball $B_{D + 1}^X(g, i, n + D)$. Put $r := D + 1$ and $v_0 := (g, i, n + D)$. Then, the $X$--ball $B_{r}^X(v_0)$ contains $K$ and is contained in $ \mathscr H_{C-1}$.

With the notation as in Lemma \ref{ofehwuil}, consider the sequence of compact sets $K_1$, \ldots, $K_m$ which collapses to the point $v_0$. Those $K_i$  are contained in $B_{r}^X(v_0) \subseteq \mathscr H_{C-1}$. We now prove that the $K_i$ are actually contained in $\mathscr H_{C}$. Indeed, $K_1 \subseteq \mathscr H_{C}$ by hypothesis. Suppose by induction that $K_i$ contains no vertices of height $C-1$, and suppose that the vertex $w_{i+1} \in K_{i+1} \setminus K_i$ has height $C-1$. Let $w_i \in K_i$ be a vertex such that $d_X(w_{i+1}, v_0) < d_X(w_i, v_0)$. Then we get a contradiction, because $d_X(w_{i+1}, v_0) \ge \height(v_0) - (C-1) \ge r$, and $d_X(w_{i}, v_0) \le r$ because $K_{i} \subseteq B_{r}^X(v_0)$.

Hence $K$ is contractible in $ \mathscr H_C$. By the arbitrariness of the compact subcomplex $K$, it follows that all homotopy groups of $ \mathscr H_C$ are trivial and the conclusion follows by Whitehead's Theorem.

\end{proof}

Notice that, in order to prove that $C$--horoballs are contractible, we have actually proved the following more precise statement.

\begin{prop} \label{wupiEACJNL}  \label{} Every compact complex $K$ in some $C$--horoball $ \mathscr H_C$ is contained in a contractible space $B_r^X(v_0) \cap \mathscr H_C$, for some $r > 0$, whose diameter in $\mathscr R_\kappa$ is linearly bounded by the diameter of $K$.
\end{prop}

\section{\label{sezione 5}Filling inequalities on $ \mathscr R_\kappa(X)$}

If $Y$ is a $CW$-complex, by $C_*(Y)$ we mean the real cellular chain complex of $Y$, i.e. the complex $H_*(Y^{(*)}, Y^{(*-1)})$ with real coefficients. We denote by $Z_k(Y)$ the subspace of cycles of $C_k(Y)$. There will be no confusion
with the notation of Section \ref{sezione 2}. Notice that, if $Y$ is a simplicial complex, the cellular chain complex $\cdots C_2(Y) \to C_1(Y) \to C_0(Y) \to \R \to 0$ is identifiable with the simplicial chain complex of oriented simplices. This is the chain complex whose $k$--th module is the real vector space generated by tuples $(y_0, \ldots, y_n)$ up to the identification
$$ (y_0, \ldots, y_i, \ldots , y_j, \ldots , y_n) = - (y_0, \ldots ,y_j, \ldots , y_i, \ldots , y_n) $$
(see \cite[Chapter 1, paragraph 5]{MUNKRES} for more details). 

We see a simplicial chain $c \in C_k(Y)$ as a finitely supported map from the set of $n$--dimensional oriented simplices of $Y$ to $\R$, and we define the support $\supp (c)$ of $c$ as the set of unoriented $n$--dimensional simplices $\Delta$ of $Y$ such that $c(\sigma) \ne 0$, where $\sigma$ is one of the two oriented simplices over $\Delta$. By $\maxh \, c$ ($\minh \, c$) we mean the height of the highest (lowest) vertex of simplices in $\supp (c)$. We denote by $\supp^{(0)}(c)$ the set of vertices that belong to some simplex in $\supp (c)$. If $A$ is a subset of $Y$ and $c = \sum_i \lambda_i \sigma_i$ is a simplicial $k$--chain, we define the restriction of $c$ to $A$ as the chain
$$ c {}_{\big| A} := \sum_{i \: \sigma_i^{(0)} \subseteq A} \lambda_i \sigma_i. $$

\subsection{A local lemma}

From now on we assume that $ \mathscr R_\kappa = \mathscr R_\kappa(X)$ satisfies the hypotheses of Corollary \ref{ohq3iulrejk}. Recall that $C$ is a fixed constant greater than $\delta$.

\begin{lem}[Local lemma] \label{lemma locale} For every $i \ge 0$, there are non--decreasing functions: 
$$ R \: \N \to \N \qquad M_\loc \: \N \times \N \to \R_{\ge 0} $$
such that, for every $D \in \N_{\ge 1}$, $v_0 \in \mathscr R_\kappa^{(0)}$ and $z \in Z_i( \mathscr R_\kappa)$ such that $\supp z \subseteq B_{D}(v_0)$, there is $a \in C_{i+1}( \mathscr R_\kappa)$ such that: 
\begin{enumerate}
    \item $\partial a = z$;
    \item $\supp a \subseteq B_{R(D)}(v_0)$; 
    \item $\|a\| \le M_\loc(D, \maxh(z)) \|z\|$;
    \item if $z$ is contained in some $C$--horoball, then $a$ is contained in the same $C$--horoball ($C$ as in Remark \ref{q0378yoehilrujwd}).
\end{enumerate}
\end{lem}

\begin{proof} Fix integers $h, D$ and $j \in I$. Let $c_1$, \ldots, $c_n$ be the collection of the $i$--dimensional simplices contained in $B_D((1, j, h))$. Let $z_1$, \ldots, $z_m$ be a basis of the subspace of cycles in $\left\langle c_1, \ldots, c_n \right\rangle_\R$, which extends bases of the spaces of cycles contained in the $C$--horoballs. We choose $a_1$, \ldots, $a_m$ so that $\partial a_1 = z_1$, \ldots, $\partial a_m = z_m$. If $z_k$ is not contained in any $C$--horoball, the chain $a_k$ may be chosen in $B_{D}((1, j, h))$, since this is contractible by Corollary \ref{ohq3iulrejk}. Otherwise, if $z_k$ is contained in some $C$--horoball, we take $a_k$ in the subcomplex $B_r^X(v_0)$ contained in that horoball, as described in Proposition \ref{wupiEACJNL}.

We extend the map $z_k \mapsto a_k$ by linearity, obtaining a linear map $\theta^{h, j, D}$ between normed spaces, where the first one is finite dimensional. Therefore $\theta^{h, j, D}$ is bounded.

Let now $z$ be a cycle in $C_i(\mathscr R_\kappa)$ with $\diam (\supp z) \le D$, and $\maxh(z) \le H$. Up to $\Gamma$--action, we may suppose that $z$ contains a vertex of the form $(1, j, h)$ for some $h \le H$, and $j \in I$. It follows that $\supp z \subseteq B_D((1, j, h))$. Then we put $a := \theta^{h, j, D}(z)$. Since $(h, j)$ is an element of the finite set $\{1, 2, \ldots, H\} \times I$, we may bound the norm of $a$ uniformly, and put $M_\loc(D, H) := \max\{\|\theta^{h, j, D}\| \: h \le H, \, j \in I\}$.
\end{proof}

\subsection{Finite sets of geodesic segments in hyperbolic spaces and filling inequalities}

The results we are going to present are inspired by the well-known fact that geodesics in hyperbolic spaces can be approximated by embedded trees (see \cite[Chapter 2]{GH}). The idea is that a set of $n$ geodesic segments \emph{resembles} a simplicial tree where all pairs of edges having a point in common diverge very rapidly from that point. In other words, the vertices of the tree are the only points near which two edges may be close to each other. Moreover, this tree is finite, and the number of vertices and edges depends only on $n$. Hence, we can split this tree into a set of balls of fixed diameter and a set of subedges that are very far from each other.

Let $k \ge 1$ and let $z$ be a $k$--dimensional cycle. If $\supp z$ is contained in an $L$--neighborhood of a set of $n$ geodesic segments, we will be able to express it as a sum of \emph{edge-cycles} and \emph{vertex-cycles}, that we can fill using the Local Lemma \ref{lemma locale} and Corollary \ref{hn0geijmovs} respectively. Therefore we will be able to fill $z$ with some control of its norm, as described in Theorem \ref{cicli in quasi grafi}. Some of the methods of this section are inspired by the proof of \cite[Lemma 5.9]{Min4}.

Let $[0, |\gamma|] \ni t \mapsto \gamma(t)$ be an arc-length parametrization of a geodesic segment $\gamma$ (where $|\gamma|$ is the length of $\gamma$) in some metric space $W$. Let $x = \gamma(t)$, for some $t \in [0, |\gamma|]$, and let $s \in \R$. By ``$\gamma(x + s)$'' we mean the point $\gamma(t+s)$, if this is defined. Otherwise, if $t+s > |\gamma|$ ($t+s < 0$) we put $\gamma(x+s) := \gamma(|\gamma|)$ ($\gamma(x+s) := \gamma(0)$). If $t < r$ and $y = \gamma(r)$, by $\gamma {}_{\Big| [x, y]}$ we mean the restriction of $\gamma$ to the interval $[t, r]$ (or its image in $W$).

\begin{comment}
\begin{rem} Let $Y$ be a simplicial complex and let $A \subseteq Y^{(0)}$ Suppose now that $z$ is a cycle. Then
$$ 0 = \partial z = \partial \Big(z {}_{\big| A}\Big) + \partial \Big(z-z {}_{\big| A}\Big). $$
In particular: $\supp \partial \Big(z {}_{\big| A}\Big) = \supp \partial \Big(z-z {}_{\big| A}\Big)$. Moreover: 
$$ \supp^{(0)} \partial \Big( z {}_{\big| A} \Big) \subseteq \supp^{(0)} z {}_{\big| A} \subseteq A, $$ 
$$ \supp^{(0)} \partial \Big(z- z {}_{\big| A} \Big) \subseteq \supp^{(0)} \Big( z - z {}_{\big| A}\Big) \subseteq \mathscr N_1(Y^{(0)} \setminus A) $$
(recall that $z$ is a simplicial chain, so each of its simplices has diameter $1$). Hence 
$$ \supp^{(0)} \partial (z {}_{\big| A}) \subseteq A \cap \mathscr N_1(Y^{(0)} \setminus A) \cap \supp^{(0)} z. $$
$\hfill \blacksquare$
\end{rem}
\end{comment} 
 
\begin{lem} \label{cicli sottili} Let $i \ge 1$. Then there are functions $R \: \N \to \N$, $D \: \N \to \N$ and $L \: \N \times \N \to \R$ which satisfy the following properties: let $z \in Z_i(\mathscr R_\kappa)$ be such that $\supp z \subseteq \mathscr N_S(\gamma)$, for some geodesic segment $\gamma$ and $S \in \N$. Then, for $R = R(S)$ and $D = D(S)$ there is an expression
$$ z = \sum_k z_k $$ 
where the $z_k$ are cycles such that
\begin{enumerate}
    \item $\supp z_k \subseteq B_R(x_k)$, where $x_k := \gamma(kD + D/2)$;
    \item \label{pq347hgureiajdks} $\sum_k \|z_k\| \le L(S, \maxh(z)) \|z\|$;
    \item if $\supp z \subseteq \mathscr H_C$ for a $C$--horoball $ \mathscr H_C$, then the same is true for every $z_k$.
\end{enumerate}
\end{lem}

\begin{proof} Take $D \ge 2S + 3$. Let $y_k := \gamma(kD)$. We put
$$ \overline z_k := z {}_{\big| B_{(k+1)D}(y_0)} - z {}_{\big| B_{kD}(y_0)}. $$
In other words, $ \overline z_k$ is the restriction of $z$ to the set of simplices contained in $B_{(k+1)D}(y_0)$ that are not contained in $B_{kD}(y_0)$. It follows immediately that $z = \sum_k \overline z_k$. Let us put: 
$$ R := D/2 + 2S \qquad r := S+1. $$ 
Notice that $D > 2S + 2 = 2r$. We have: 
\begin{equation} \label{p3ufheqw} \supp \overline z_k \subseteq \mathscr N_S(\gamma) \cap \left(B_{(k+1)D}(y_0) \setminus B_{kD-1}(y_0) \right) \subseteq B_{D/2 + 2S}(x_k) =  B_{R}(x_k). \end{equation}
In fact, let $v$ be a vertex in $\mathscr N_S(\gamma) \cap \left(B_{(k+1)D}(y_0) \setminus B_{kD-1}(y_0) \right)$. Let $x \in \gamma$ be such that $d(v, x) \le S$. Notice that $x \in B_{(k+1)D+S}(y_0) \setminus B_{kD-1-S}(y_0)$, i.e. $kD - S \le d(y_0, x) \le (k+1)D + S$. Hence $d(x, x_k) \le D/2 + S$, and $d(v, x_k) \le d(v, x) + d(x, x_k) \le D/2 + 2S$, whence the second inclusion in \eqref{p3ufheqw} follows. It follows from \eqref{p3ufheqw} that 
\begin{equation} \label{uofwh} \| \overline z_k\| \le \|z {}_{\big| B_{R}(x_k)} \|. \end{equation}
Now, from the first inclusion of \eqref{p3ufheqw} we get
$$ \supp^{(0)} (\partial \overline z_k) \subseteq $$ 
$$ \mathscr N_S(\gamma) \cap \left(\left\{x \in \mathscr R_\kappa^{(0)}(X) \: kD \le d(y_0, x) \le kD+1\right\} \sqcup \right. $$ $$ \phantom{mmmmmmmm} \left. \sqcup \left\{x \in \mathscr R_\kappa^{(0)}(X) \: (k+1)D - 1 \le d(y_0, x) \le (k+1)D \right\}\right) \subseteq  $$ $$ \subseteq B_{S+1}(\gamma(kD)) \sqcup B_{S+1}(\gamma((k+1)D)) = B_{r}(y_k) \sqcup B_{r}(y_{k+1}). $$
% il supporto del bordo è circa il bordo del supporto
Therefore, since the last two subcomplexes are disjoint, we can put 
$$ \partial \overline z_k = b_k' + b_k \qquad \supp b_k' \subseteq B_r(y_{k}), \quad \supp b_k \subseteq B_r(y_{k+1}). $$
Notice that $\|b_k'\| + \|b_k\| = \|b_k' + b_k\| \le (i+1) \| \overline z_k\| \le (i+1)\|z {}_{\big| B_R(x_k)}\|$. We have
$$ 0 = \partial z = \sum_k \partial \overline z_k = \sum_k b_k + b_k' = \sum_k b_{k} + b_{k+1}'. $$
By looking at supports, we note that it follows that $b_k = - b_{k+1}'$. Since $b_k' + b_k$ is a cycle (in the augmented simplicial chain complex of $ \mathscr R_\kappa$) and $b_k$ and $b_k'$ have disjoint supports, it follows that $b_k$ and $b_k'$ are cycles too, \emph{if} their dimension is at least $1$. The same is true if the $b_k^{(')}$ are $0$--dimensional. Indeed, it is easy to see that $b_0' = 0$, hence $b_0 = b_1'$ is a cycle. By induction, if $b_k'$ is a cycle, it follows that $b_k = b_{k+1}'$ is a cycle too. Hence all the $b_k^{(')}$ are cycles.

We fill $b_k$ and $b_k'$ by $a_k'$ and $a_k$ using the local lemma, and we also require that $a_k' = -a_{k-1}$. Since $b_k$ and $b_k'$ have diameter bounded by $2r$, by the local lemma we have a function $L(S, \cdot) := M_\loc(2r, \cdot) = M_\loc(2(S+1), \cdot)$ such that $\|a_k\| \le L(S, \maxh(b_k)) \|b_k\|$. If $H = \maxh(z)$, then
\begin{equation} \label{ertyhgf} \|a_k\| \le L(S, \maxh(b_k)) \|b_k\| \le L(S, H) (i+1) \|z {}_{\big| B_R(x_k)}\|. \end{equation}
Hence also

\begin{equation} \label{ertyhgf1} \|a_k'\| = \|a_{k-1}\| \le L(S, H) (i+1) \|z {}_{\big| B_R(x_{k-1})}\|. \end{equation}

We put
$$ z_k := \overline z_k - a_{k}' - a_k. $$
By \eqref{uofwh}, \eqref{ertyhgf} and \eqref{ertyhgf1}, there is a function $L' \: \N \times \N \to \R$ such that
$$ \|z_k\| \le L'(S, H) \|z {}_{\big| B_{R+D}(x_k)} \|. $$
We have $ \partial z_k = b_k + b_k' - b_{k}' - b_k = 0$, and
$$ \sum_k z_k = \sum_k \overline z_k - a_{k}' - a_k = \sum_k \overline z_k - \sum_k \overline a_{k}' + a_k = z - \sum_k \overline a_{k}' + a_{k-1} = z. $$
Finally, since the balls $B_{R+D}(x_k)$ and $B_{R+D}(x_{k+5})$ have disjoint supports (because 
$4S < 2D \Rightarrow 2R = D + 4S \le 3D \Rightarrow 2(R + D) \le 5D$), we have
$$ \sum_k \|z {}_{\big| B_{R+D}(x_k)} \| = \sum_{j=0}^4 \Big\| \sum_{k = j \mod 5} z {}_{\big| B_{R+D}(x_k)} \Big\| \le 5 \|z\|. $$

Therefore, Condition (2) in the statement holds with $L(S, H) = 5 L'(S, H)$. Finally, (3) follows from the local lemma.
\end{proof}

\begin{cor} \label{hn0geijmovs} For every $i \in \N$ there are functions $S' \: \N \to \N$ and $M_\thin \: \N \times \N \to \R$ such that, for every geodesic segment $\gamma$ and every cycle $z \in Z_i( \mathscr R_\kappa)$ with $\supp z \subseteq \mathscr N_S(\gamma)$ for some $S \ge 0$, there is a filling $a$ of $z$ with $\supp a \subseteq \mathscr N_{S'}(\gamma)$ and such that
$$ \| a \| \le M_\thin(S, \maxh(z)) \|z\|. $$
Moreover, we may impose that $a$ is contained in a $C$--horoball $ \mathscr H_C$, if the same is true for $z$ ($C$ is as in Remark \ref{q0378yoehilrujwd}).
\end{cor}

\begin{proof} Split $z$ as the sum of the cycles $z_k$ be as in the previous lemma. Now, let $R = R(S)$ as in the previous lemma. If $\supp z \cap B_R(x_k) \ne \emptyset $, we have $\maxh (z_k) \le \maxh (B_R(x_k))$; otherwise it is clear from the construction of $z_k$ that $z_k = 0$. In any case we have $\maxh (z_k) \le \supp (z) + 2R$. Moreover, by (1) of the previous lemma, $\max \, \diam (z_k) \le 2R$.
Fill $z_k$ with $a_k$ as in the local lemma, and put $a = \sum_k a_k$. Let $M_\loc \: \N \times \N \to \R$ be as in the local lemma. Hence
$$ \|a\| \le \sum_k \|a_k\| \le$$ $$\le M_\loc(\max \diam (z_k), \maxh(z) +2R) \sum_k \|z_k\| \le $$ 
$$ \le M_\loc(2R, \maxh(z)+2R) \,\, L(S, \maxh(z)) \|z\|. $$ 
So we can put $M_\thin(S, h) := M_\loc(2R, h + 2R) L(S, h)$.
\end{proof}

The next lemma holds for every $\delta$--hyperbolic space $X$.

\begin{lem} \label{quasiGrafi} Let $\alpha_1$, \ldots, $\alpha_n$ be $n$ geodesic segments. Then, for every $S \in \N$, there exist constants $R = R(S, n)$, $ {p}{} = {p}{}(n)$, ${q}{} = {q}{}(n)$, points $x_1$, \ldots, $x_{{p}{}} \in X$ and geodesic segments $\gamma_1$, \ldots, $\gamma_{{q}{}}$ such that 
$$ \bigcup_{k=1}^n \alpha_k \subseteq \bigcup_{i=1}^{{p}{}} B_R(x_i) \cup \bigsqcup_{j=1}^{{q}{}} \mathscr N_{2\delta}(\gamma_j) $$
where the $\gamma_j$ are $S$--far from each other.
\end{lem}

\begin{proof} We prove the statement by induction on $n$. The case $n=1$ is obvious. Suppose that the statement is proved for $n-1$ segments, and put
$$ {p}{} = {p}{}(n-1), \qquad {q}{} = {q}{}(n-1). $$ 
Hence we have balls $B_R(x_1)$, \ldots, $B_R(x_{{p}{}})$ and geodesic segments $\gamma_1$, \ldots, $\gamma_{{q}{}}$ associated with $\alpha_1$, \ldots, $\alpha_{n-1}$ as in the statement. We fix an orientation on $\alpha_n$ and for every $1 \le j \le {q}{}$ such that $d(\gamma_j, \alpha_n) \le S$ (here $d$ denotes the distance between sets) we denote by $x_j$ (resp. $y_j$) the first (resp. the last) point on $\alpha_n$ such that $d(x_j, \gamma_j) \le S$, $d(y_j, \gamma_j) \le S$. By hyperbolicity, it is easy to see that 
$$ \alpha_n {}_{\big| [x_j + S + \delta, y_j - S - \delta]} \subseteq \mathscr N_{2 \delta} (\gamma_j) $$
(some of these intervals may be empty). Since the $\gamma_j$ are $(2S + 6 \delta + 1)$--far from each other, we claim that, up to reindexing, we have
$$ x_1 \le y_1 \le x_2 \le y_2 \le \cdots \le x_{k} \le y_{k} $$
for $k \le q$. Indeed, $x_j \le y_j$ by definition. Moreover, the points between $x_j$ and $y_j$ are $(S+3 \delta)$--close to $\gamma_j$. Since there cannot be points in $X$ that are ($S+3 \delta$)--close to two different $\gamma_j$, we have $[x_j, y_j] \cap [x_k, y_k] = \emptyset$, for $j \ne k$, whence the claim.

The segments of type $\alpha_n {}_{\big| [y_j + S + \delta, \, x_{j+1} - S - \delta]}$ (where by $y_0$ and $x_{k+1}$ we mean the left and right extreme of $\alpha_n$ respectively) are $S$--far from all the $\gamma_j$ and ($2S + 2 \delta$)--far from each other. Adding to the $\gamma_j$ the segments of type $\alpha_n {}_{\big| [y_j + S + \delta,  x_{j+1} - S - \delta]}$ and to the $B_R(x_i)$ the balls of type $B_{S + \delta}(x_j)$, $B_{S + \delta}(y_j)$ we complete the inductive step.
\end{proof}

We need in Theorem \ref{cicli in quasi grafi} a stronger version of the lemma above in order to deal with $1$--dimensional cycles. 
%Suppose that the centers $x_k$ of the balls of Lemma \ref{quasiGrafi} are $2q(R+S)$ far apart from each other, with $S > q\delta$. 
In the notation of Lemma \ref{quasiGrafi}, we say that two distinct balls $B_1$ and $B_2$ are \emph{linked} if there is a $\gamma_j$ such that $d(\gamma_j, B_1) \le S$, $d(\gamma_j, B_2) \le S$ and, if $v_1, v_2 \in \gamma_j$ are such that $d(v_1, B_1) \le S$, $d(v_2, B_2) \le S$, there is no point $v_3 \in \gamma_j$ between $v_1$ and $v_2$ such that $d(v_3, B_3) \le S$, for some ball $B_3$ distinct from $B_1$ and $B_2$. We thus get a graph structure on the balls of Lemma \ref{quasiGrafi}.

For any $r \in \N$, we call \emph{$r$--cycle} a sequence $\{B_u\}_{u \in \Z/r\Z}$ of $r$ distinct balls, with $B_u$ linked to $B_{u+1}$ for all $u \in \Z/r\Z$. In the following lemma we prove that, if the balls are sufficiently far apart, there are no $r$--cycles for $r \ge 3$. Hence the graph is a forest, i.e. a graph which is a disjoint union of trees. In particular, we will be able to talk about \emph{leaf-balls}, i.e. balls that correspond to vertices that are ends of at most one edge. Notice that, in the conditions of Lemma \ref{quasiGrafiPotenziato}, for any pair of balls $B_1$ and $B_2$, there is at most one $\gamma_j$ such that $d(\gamma_j, B_1) \le S$ and $d(\gamma_j, B_2) \le S$.

%sostituire q-->p
\begin{lem} \label{quasiGrafiPotenziato} Suppose that we have an inclusion
$$ \bigcup_{k=1}^n \alpha_k \subseteq \bigcup_{u=1}^{{p}{}} B_R(x_u) \cup \bigsqcup_{j=1}^{{q}{}} \mathscr N_{2\delta}(\gamma_j) $$
where the $\alpha_k$ and $\gamma_j$ are geodesic segments, and the $\gamma_j$ are $S$--far apart, for some $S > (p{}+6) \delta$.

Moreover, suppose that the $x_u$ are $2p{}(R+S)$-far apart. Then there are no $r$--cycles, for any $r \ge 3$. \end{lem} 

\begin{proof} Up to reindexing, we may suppose that the balls $B_R(x_1)$, .., $B_R(x_r)$ constitute an $r$--cycle. Put $B_u :=B_R(x_u)$. We slightly abuse notation by identifying the natural numbers $1$, .., $r$ with the corresponding elements of $\Z/r\Z$. Let $l_u$ be the minimal subsegment of some $\gamma_j$ such that the ends of $l_u$ are $S$--close to $B_u$ and $B_{u+1}$ respectively.

In the following, we denote by $[x_u, x_{u+1}]'$ the subsegment of $[x_u, x_{u+1}]$ which is outside the balls $B_{R+S}(x_u)$ and $B_{R+S}(x_{u+1})$. By $\delta$--hyperbolicity, the Hausdorff distance between $l_u$ and $[x_u, x_{u+1}]'$ is at most $4 \delta$. For a geodesic $r$--agon in a $\delta$--hyperbolic space, any edge is contained in the $(r-2)\delta$--neighborhood of the union of the other edges. Hence $[x_u, x_{u+1}] \subset \bigcup_{k \ne u} \mathscr N_{(r-2) \delta}([x_k, x_{k+1}])$, therefore
$$ [x_u, x_{u+1}]' \subset \bigcup_{k \ne u} \mathscr N_{(r-2) \delta}([x_k, x_{k+1}]') \cup \bigcup_{k \ne u} B_{R+S+(r-2)\delta}(x_k) $$
Since the length of $[x_u, x_{u+1}]'$ is at least $2p{}(R+S) - 2(R+S) \gneq (r-1) (R+S+(r-2)\delta)$, it follows that the $r-1$ balls $B_{R+S+(r-2)\delta}(x_k)$ can't cover all of $[x_u, x_{u+1}]'$. Hence there is some $k \ne u$ such that $d([x_u, x_{u+1}]', [x_k, x_{k+1}]') \le (r-2) \delta$. Since the Hausdorff distance between $l_u$ and $[x_u, x_{u+1}]'$ ($l_k$ and $[x_k, x_{k+1}]'$) is at most $4 \delta$, it follows that the distance between $l_u$ and $l_k$ is less than $(r+6)\delta \le (p{}+6)\delta< S$, whence the contradiction.
\end{proof}

We now consider the problem of filling cycles whose supports are close to geodesic segments. 

% i è la dimensione, i non c'era prima nello statement, quindi al posto di i usiamo u, nalla dimostrazione
% quindi k non è più la dimensione. Però stupidamente veniva usata sia come indice (\sum e \bigcup), sia come dimensione, quindi ora la usiamo solo come indice

\begin{theorem} \label{cicli in quasi grafi} Let $n,\,i,\,L \in \N$, $i \ge 1$, and let $C \in \N$ be as in Remark \ref{q0378yoehilrujwd}. Then there exists $L' = L'(n, i, L) \in \N$ such that, for every cycle $z \in Z_i( \mathscr R_\kappa)$ and every family of geodesic segments $\alpha_1$, \ldots, $\alpha_n$ such that $\supp z \subseteq \mathscr N_L (\alpha_1 \cup \ldots \cup \alpha_n)$, there exists $a \in C_{i+1}( \mathscr R_\kappa)$ with $ \partial a = z$ and

\begin{equation} \label{primo} \supp a \subset \mathscr N_{L'} (\supp z). \end{equation}

In particular, up to increasing $L'$, we have $\supp a \subseteq \mathscr N_{L'} (\alpha_1 \cup \ldots \cup \alpha_n)$, and $\maxh(a) \le \maxh(z) + L'$. Moreover, there exists a function $M_{} = {M_{}\,}(n, i, L) \: \N \to \N$ such that
\begin{equation} \label{secondo} \|a\| \le M_{}(\maxh(z)) \|z\|. \end{equation}

Finally, we can require that, if $z$ is contained in some $C$--horoball, $a$ is contained in the same $C$--horoball.
\end{theorem}

\begin{proof} Let $\N \ni S = 2L + 4\delta +1$. Let $R = R(S, n)$, ${p}{} = {p}{}(n)$ and ${q}{} = {q}{}(n)$ be as in Lemma \ref{quasiGrafi}, in such a way that for some vertices $x_u$ and geodesic segments $\gamma_j$ which are $S$--far from each other
\begin{equation} \label{n00afpw} \bigcup_k \alpha_k \subseteq \bigcup_{u=1}^{{p}{}} B_R(x_u) \cup \bigsqcup_{j=1}^{{q}{}} \mathscr N_{2 \delta}(\gamma_j). \end{equation}

Let $z$ be a cycle whose support is contained in an $L$--neighborhood of the $\alpha_u$. Hence

\begin{equation} \label{9034y85efios1} \supp z \subseteq \bigcup_{u=1}^{{p}{}} B_{R+L}(x_u) \cup \bigsqcup_{j=1}^{{q}{}} \mathscr N_{2\delta+L}(\gamma_j). \end{equation}
The fact that the $ \mathscr N_{2 \delta + L}(\gamma_j)$ are pairwise disjoint is a consequence of our requirements on $S$. 

By suitably choosing a subset $I$ of $\{1, \ldots, {p}{}\}$, we get that there exists $R + S + 1 \leq R' \le (2p{}+1)^{{p}{}} (R+S+1)$ such that the $x_{u}$, $u \in I$, are $2p{}(R'+S)$ far apart, and $ \bigcup_{u=1}^{{p}{}} B_{R+S+1}(x_u) \subseteq \bigsqcup_{u \in I} B_{R'}(x_{u})$. Indeed, the case ${p}{} = 1$ is trivial. Otherwise, if two balls $B_{R+S+1}(x_{u_1})$ and $B_{R+S+1}(x_{u_2})$ are not $2p{}(R+S)$ far apart, we consider the balls $B_{(2p{}+1)(R+S+1)}(x_{u})$, for all $u \ne u_2$. We have that $B_{R+S+1}(x_{u_1}) \cup B_{R+S+1}(x_{u_2}) \subseteq B_{2p{}(R+S+1)}(x_{u_1})$. Then we continue by reverse induction on ${p}{}$.

We have
\begin{equation} \label{2roqohuirel} \supp z \subseteq \bigcup_{u=1}^{{p}{}} B_{R+L+1}(x_u) \cup \bigsqcup_{j=1}^{{q}{}} \mathscr N_{2\delta+L}(\gamma_j) \subseteq \bigsqcup_{u \in I} B_{R'}(x_{u}) \cup \bigsqcup_{j=1}^{{q}{}} \mathscr N_{2\delta+L}(\gamma_j). \end{equation}
Put 
$$ z' := z {}_{\big| \bigsqcup_{j=1}^N \mathscr N_{2\delta+L}(\gamma_j)}. $$
There is a unique expression
$$ z' = \sum_{j=1}^{{q}{}} z_{\gamma_j} \qquad \textrm{where } \supp z_{\gamma_j} \subseteq \mathscr N_{2\delta+L}(\gamma_j). $$ 
By \eqref{2roqohuirel} and the definition of $z'$ and $ \bigcup_{u=1}^{{p}{}} B_{R+L+1}(x_u) \subseteq \bigsqcup_{u \in I} B_{R'}(x_{u})$  we get
$$ \supp \partial z_{\gamma_j} \subseteq \bigsqcup_{u \in I} B_{R'}(x_{u}). $$ 
Hence we can put: 
\begin{equation} \label{w8or7js} \partial z_{\gamma_j} = \sum_{u \in I} b_j^{u} \qquad \supp b_j^{u} \subseteq B_{R'}(x_{u}) \end{equation} 
(this expression being unique). 

Suppose that $i \ge 2$. Then, by \eqref{w8or7js} and the disjointness of the $B_{R'}(x_{u})$, $u \in I$, the $b_j^{u}$ must all be cycles. 

The same is true if $i=1$. Indeed, let $B_{R'}(x_u)$ be a leaf-ball as in Lemma \ref{quasiGrafiPotenziato}. Fix a $\gamma_j$. There are three possibilities. If $d(\gamma_j, B_{R'}(x_{u})) > S$ then $b_j^u = 0$. If $d(\gamma_j, B_{R'}(x_{u})) \le S$ and $d(\gamma_j, B_{R'}(x_{k})) > S$ for every $k \ne u$, then $b_j^u = \partial z_{\gamma_j}$. In both cases $b_j^u$ is a cycle. Suppose now that there is some ball $B_{R'}(x_k)$, $k \ne u$ such that $d(\gamma_j, B_{R'}(x_k)) \le S$, $d(\gamma_j, B_{R'}(x_u)) \le S$. By definition of leaf-ball, there is at most one such $j$. For any $u$, the sum $b_j^u + \sum_{r \ne j} b_r^u$ is a cycle. Since the $b_r^u$ in the sum are all cycles, it follows that $b_j^u$ is a cycle too. Hence in any case, if $u$ corresponds to a leaf-ball, all $b_j^u$ are cycles. By an inductive argument, we can apply the same line of reasoning to the balls that are connected to leaf-balls. It follows that all $b_j^u$ are cycles.

Let $a_{j}^{u}$, $u \in I$, be such that $ \partial a_{j}^{u} = b_{j}^{u}$ as in the local lemma. By definition of $z'$, 
$$\supp (z-z') \subseteq \bigsqcup_{u \in I} B_{R'}(x_{u}). $$

For $u \in I$, let $z_{u}$ be the restriction of $z-z'$ to $B_{R'}(x_{u})$. Then
$$ z-z' = \sum_{u \in I} z_{u} $$
and
$$ 0 = \partial z = \partial (z-z') + \partial z' = \sum_{u \in I} \left( \partial z_{u} + \sum_{j=1}^{{q}{}} b_j^{u}\right) = \sum_{u \in I} \partial \left( z_{u} + \sum_{j=1}^{{q}{}} a_j^{u}\right) $$
$$ \Rightarrow \partial \left(z_{u} + \sum_{j=1}^{{q}{}} a_j^{u} \right) = 0 \qquad \forall u \in I, $$
the implication being true because the $B_{R'}(x_{u})$ are disjoint. The chain $ \overline z_{\gamma_j} = z_{\gamma_j} - \sum_{u \in I} a_j^{u}$ is a cycle (by \eqref{w8or7js}), and $ \overline {z_{u}} = z_{u} + \sum_{j=1}^{{q}{}} a_j^{u}$, for $u \in I$, is also a cycle by the equality above. By summing, we get
$$ \sum_{j=1}^{{q}{}} \overline {z}_{\gamma_j} + \sum_{u \in I} \overline{z}_u = \sum_{j=1}^{{q}{}} \left(z_{\gamma_j} - \sum_{u \in I} a_j^{u} \right) + \sum_{u \in I} \left(z_{u} + \sum_{j=1}^{{q}{}} a_j^{u}\right) = $$ $$ = \sum_{j=1}^{{q}{}} z_{\gamma_j} + \sum_{u \in I} z_{u} - \sum_{j=1}^{{q}{}} \sum_{u \in I} a_j^{u} + \sum_{u \in I} \sum_{j=1}^{{q}{}} a_j^{u} = \sum_{j=1}^{{q}{}} z_{\gamma_j} + \sum_{u \in I} z_{u} = z. $$

For all $u$, $\supp z_u \cup \bigcup_j \supp b_u^j \subseteq B_{R'}(x_u)$. By the local lemma, there is a constant $R'' = R''(R')$ such that $\bigcup_j \supp a_u^j \subseteq B_{R''}(x_u)$, hence $\supp \overline z_u \subseteq B_{R''}(x_u)$ too. Analogously, we have $\supp z_{\gamma_j} \cup \bigcup_u b_j^u \subseteq \mathscr N_{2\delta + L}(\gamma_j)$, hence also $\supp \overline z_{\gamma_j} \subseteq \supp z_{\gamma_j} \cup \bigcup_u a_j^u \subseteq \mathscr N_{S'}(\gamma_j)$, where we can put $S' = \max \left\{2\delta + L, \, R''\right\}$. 

We fill the $\overline z_u$ and the $\overline z_{\gamma_j}$ by $a_u$ and $a_{\gamma_j}$ as in the local lemma and Corollary \ref{hn0geijmovs} respectively, and put
$$ a := \sum_{u \in I} a_u + \sum_{j=1}^{{q}{}} a_{\gamma_j}. $$
By the local lemma again, the filling $a_u$ of $\overline z_u$ has support contained in some $B_{R'''}(x_u)$, where $R'''$ only depends on $R''$. Finally, by Lemma \ref{hn0geijmovs}, we get that $\supp a_{\gamma_j} \subseteq \mathscr{N}_{S''}(\gamma_j)$, for some $S''$ which only depends on $S'$. Hence Condition \eqref{primo} is easily verified, and we can put $L' = \max\{S'',\, R'''\}$.

In order to check the condition about the horoballs note that, if $z$ is contained in some $C$--horoball, then all the $z_{\gamma_j}$ and $z_{u}$ are contained in the same $C$--horoball. Hence, by (4) in the local lemma and (3) in Corollary \ref{hn0geijmovs}, the same is true for the $a_u$ and the $a_{\gamma_j}$. 

We are finally left to prove \eqref{secondo}. Let 
$$ K := \max \{M_\thin(S'', \maxh(z) + S''), \, M_\loc(2 R'', \maxh(z))\}, $$
where $M_\thin$ is the function of Corollary \ref{hn0geijmovs} and $M_\loc \: \N \times \N \to \R$ is the function of Point (3) of the local lemma. Then
$$ \|a\| \le \sum_{u \in I} \|a_u\| + \sum_{j=1}^{ {q}{}} \|a_{\gamma_j}\| \le K \left(\sum_{u \in I} \| \overline z_u\| + \sum_{j=1}^{ {q}{}} \| \overline z_{\gamma_j}\| \right), $$ 
$$ \sum_u \|\overline z_u\| \le \sum_u \| z_u\| + \sum_{uj} \|a_{j}^u\|, \qquad \sum_j \|\overline z_{\gamma_j}\| \le \sum_j \| z_{\gamma_j}\| + \sum_{uj} \| a_{j}^u \|. $$

By disjointness of the supports of the $z_u$ and the $\overline z_{\gamma_j}$ we get
$$ \sum_u \|z_u\| = \|\sum_u z_u\| \le \|z\| \qquad \sum_j \| \overline z_{\gamma_j}\| = \| \sum_j \overline z_{\gamma_j}\| \le \|z\|. $$
Now, from the construction of the $b_j^u$, we get $\maxh(b_j^u) \le \maxh (z)$. Since the $a_j^u$ fill the 
$b_j^u$ as in the local lemma, we get
%$$ \sum_{ij} \| a_{j}^i \| \le R_{k?}(\max \diam_{ij} a_{ij}) \sum_{ji} \| b_{j}^i \|, $$
$$ \sum_{uj} \| a_{j}^u \| \le M_{\loc}(R', \maxh(z)) \sum_{ju} \| b_{j}^u \|, $$
because the $b_j^u$ are contained in balls of radius $R'$.
\end{proof}

\section{\label{sezione 6} Proof of part \textrm{$(a)$} of Theorem \ref{main MIO theorem}}

The following homological lemma helps us to outline the strategy we intend to pursue in order to prove Theorem \ref{main MIO theorem} (a).

\begin{lem}[Homological lemma] \label{il lemma omologico o3827y} Let $(\Gamma, \Gamma')$ be a group-pair. Let $\St_*$ and $\St_*'$ be as in Definition \ref{standard resolutions}. The augmented complexes $\St_*^+ \:= \St_* \to \R \to 0$ and $\St_*'^+ \:= \St_* \to \R(\Gamma / \Gamma') \to 0$ are $\Gamma$--projective resolutions of $\R$ and $\R(\Gamma / \Gamma')$. In general, by a \emph{map between resolutions} of the same $\Gamma$--module $M$ we mean a chain $\Gamma$--map that extends the identity of $M$. Let $\varphi_i \: \St_* \to \St_*$, $i = 1, 2$ be chain $\Gamma$--maps which satisfy the following hypotheses: \begin{enumerate}
	\item $\varphi_i$ extends to a map between resolutions $\varphi_i^+ \: \St_*^+ \to \St_*^+$;
    \item $\varphi_i$ restricts to a map $\varphi_i' \: \St_*' \to \St_*'$;
    \item $\varphi_i'$ extends to a map between resolutions $\varphi_i'^+ \: \St_*'^+ \to \St_*'^+$.
\end{enumerate}
Then there is a $\Gamma$--equivariant homotopy $T$ between $\varphi_1^+$ and $\varphi_2^+$ that restricts to a homotopy between $\varphi_1'^+$ and $\varphi_2 '^+$ (in $\St_*'$). Given a $\Gamma$--module $V$, the dual maps $\varphi^1$ and $\varphi^2$ of $\varphi_1$ and $\varphi_2$ induce homotopically equivalent maps on the complex $\hom^\Gamma(\St_*/\St_*', V) =: \St^{\rel \, *}(\Gamma, \Gamma'; V)$, for every $\Gamma$--module $V$.
\end{lem}

We will apply the homological lemma to the diagram

\begin{equation} \label{0y45er}
\begin{array}{ccccccccccccccccccccccccccccccc}
\vdots  && \vdots && \vdots \\
\scriptstyle{\Big\downarrow}  && \scriptstyle{\Big\downarrow}  && \scriptstyle{\Big\downarrow} \\
\St_2(\Gamma) & \overset{\psi_2}{ \longrightarrow}  & C_2( \mathscr R_\kappa) & \overset{\varphi_2}{ \longrightarrow} & \St_2(\Gamma) \\
\scriptstyle{\Big\downarrow}  && \scriptstyle{\Big\downarrow}  && \scriptstyle{\Big\downarrow} \\
\St_1(\Gamma) & \overset{\psi_1}{ \longrightarrow}   & C_1(\mathscr R_\kappa) & \overset{\varphi_1}{ \longrightarrow} & \St_1(\Gamma) \\
\scriptstyle{\Big\downarrow}  && \scriptstyle{\Big\downarrow}  && \scriptstyle{\Big\downarrow} \\
\St_0(\Gamma) & \overset{\psi_0}{ \longrightarrow} & C_0(\mathscr R_\kappa) & \overset{\varphi_0}{ \longrightarrow} & \St_0(\Gamma) \\
\scriptstyle{\Big\downarrow}  && \scriptstyle{\Big\downarrow}  && \scriptstyle{\Big\downarrow} \\
\R & \overset{\id}{ \longrightarrow} & \R & \overset{\id}{ \longrightarrow} & \R \\ 
\scriptstyle{\Big\downarrow}  && \scriptstyle{\Big\downarrow}  && \scriptstyle{\Big\downarrow} \\
0 && 0 && 0, \\
\end{array}
\end{equation}
where $\mathscr R_\kappa = \mathscr R_\kappa(X)$ and $\kappa \ge 4 \delta + 6$ as in Corollary \ref{ohq3iulrejk}.

We wish to prove that the composition $\varphi_* \circ \psi_*$ satisfies the hypotheses of the homological lemma, and that $\psi^n \circ \varphi^n (f) = f \circ \varphi_n \circ \psi_n$ is a bounded cocycle for every $n \ge 2$ and for every cocycle $f \in \hom^\Gamma(\St_n; V)$. This will prove the surjectivity of the comparison map since, by Lemma \ref{il lemma omologico o3827y}, for any given cocycle $f$, the cocycle $f \circ \varphi_n \circ \psi_n$ is cobordant to $f$ and bounded.

In order to fulfill conditions (1), (2), (3) of the homological lemma it is sufficient to find $\Gamma$--equivariant chain maps $\varphi_*$ and $\psi_*$ such that $\psi_*$ maps simplices in $\St'$ into simplices in the corresponding $C$--horoballs of $ \mathscr R_\kappa$, and vice versa for $\psi_*$. 

We now define $\varphi_*$. If $i \ge 1$ we put $\varphi_0(g, i, n) := (g, i)$. Otherwise, we define $\varphi_0(g, 0) := \frac{1}{|I|} \sum_{i \in I} (g, i)$. For $i \ge 1$ and an $i$--dimensional simplex $[x_0, \ldots, x_i]$ of $ \mathscr R_\kappa(X)$ we put

\begin{equation} \label{3qiuhlefjwkm} \varphi_i([x_0, \ldots, x_i]) := \frac{1}{(1+i)!} \sum_{\pi \in S_{i+1}} \varepsilon(\pi) (\varphi_0(x_{\pi(0)}), \ldots, \varphi_0(x_{\pi(i)})), \end{equation}
where $S_{i+1}$ is the group of permutations of $\left\{0, \ldots, i\right\}$, and $ \varepsilon(\pi) = \pm 1$ is the sign of $\pi$. The apparently cumbersome definition of the map $\varphi_*$ follows from the fact that in $C_*( \mathscr R_\kappa)$ we have oriented simplices, whereas in $\St_*$ we have \emph{ordered} ones, and that the action of $\Gamma$ on $ \mathscr R_\kappa(X)$ may map a simplex to itself, changing the order of the vertices.

Much more effort will be needed for the definition of $\psi_*$, to which the rest of this section is dedicated. The fundamental tool that we will use is the bicombing defined in \cite{GM}.

\begin{dfn}\cite[Section 3]{Min1} \label{ } Given a group $\Gamma$ acting on a graph $G$ through simplicial automorphisms, a \textbf{homological bicombing} is a function
$$ q \: G^{(0)} \times G^{(0)} \to C_1(G) $$
such that $\partial q(a, b) = b - a$ for all $(a, b) \in G^{(0)} \times G^{(0)}$. We say that $q$ is \textbf{antisymmetric} if $q(a, b) = - q(b, a)$ for all $a, b \in G^{(0)}$, and \textbf{$\Gamma$--equivariant} if $\gamma q(a, b) = q(\gamma a, \gamma b)$ for all $\gamma \in \Gamma$ and $a, b \in G^{(0)}$. Moreover, $q$ is \textbf{quasi--geodesic} if there is a constant $D > 0$ such that, for all $a$, $b \in G^{(0)}$: \begin{enumerate}
	\item $\|q(a, b)\| \le D d(a, b)$;
	\item $\supp q(a, b) \subseteq \mathscr N_D([a, b])$.
\end{enumerate}
We note that, if $G$ is a hyperbolic graph, the precise choice of a geodesic $[a, b]$ between $a$ and $b$ is, up to increasing the constant $D$, irrelevant.
\end{dfn}

The homological bicombing $Q$ in the following theorem is based on the bicombing constructed by Mineyev in \cite{Min1}. The relevant properties of $Q$ are described in \cite[Section 5]{GM} and \cite[Theorem 6.10]{GM}.

\begin{theorem} \label{main theorem GM} If $(\Gamma, \Gamma')$ is a relatively hyperbolic pair, there is a bicombing $Q$ on the associated cusped space $X$ such that: \begin{enumerate}
	\item $Q$ is quasi-geodesic
	\item $Q$ is $\Gamma$--equivariant;
	\item $Q$ is antisymmetric;
	\item there is $K > 0$ such that, for all $a$, $b$, $c \in X^{(0)}$, there are $1$--cycles $z = z(a, b, c)$ and $w = w(a, b, c)$ such that
\begin{itemize}
	\item $Q(\partial (a, b, c)) = z + w$
	\item $\minh(w) \ge C > \delta$;
	\item $\|z\| \le K$;
	\item $\maxh (z) \le K$;
	\item for all $\gamma \in \Gamma$, $z (\gamma a, \gamma b, \gamma c) = \gamma z (a, b, c)$ and $w (\gamma a, \gamma b, \gamma c) = \gamma w (a, b, c)$;
	\item $z$ and $w$ are contained in the $K$--neighborhood of $[a, b] \cup [b, c] \cup [c, a]$.
\end{itemize}
\end{enumerate}
\end{theorem}

\begin{rem} \label{i3qrhiwul} Groves and Manning allow multiple edges in their definition of cusped graph (as already noted in Remark \ref{3iyewhilnj}). However, it is easy to see that, if $\overline X$ is the simplicial graph obtained by identifying edges of $X$ with the same endpoints, the obvious bicombing induced by $Q$ on $ \overline X$ satisfies all of the properties of Theorem \ref{main theorem GM}. See also \cite[Remark 6.12]{GM}
\end{rem}

We want to find a decomposition $\{\psi_k = z_k + w_k\}_{k \ge 2} \: \St_*(\Gamma) \to C_*( \mathscr R_\kappa)$ of $\psi_k$ into $\Gamma$--equivariant chain maps: $\{z_k\}_{k \ge 2}$ and $\{w_k\}_{k \ge 2}$ such that
\begin{enumerate}
	\item[$(A)$] $\|z_k(\Delta)\|$ is uniformly bounded independently of $\Delta \in \St_k$;
	\item[$(B)$] $\maxh(z_k(\Delta))$ is uniformly bounded independently of $\Delta \in \St_k$;
	\item[$(C)$] $\minh(w_k(\Delta)) \ge C$ for every $\Delta \in \St_k$; 
	\item[$(D)$] $z_*$ and $w_*$ map elements in the basis of $\St'$ into $C$--horoballs.
\end{enumerate}

We now show how the conclusion follows from the existence of a map $\psi_*$ satisfying the four conditions above, and then we construct such a $\psi_*$. It is easy to see that, if an $i$--dimensional simplex $s$ of $ \mathscr R_\kappa(X)$ is not contained in a single $C$--horoball, it must satisfy $\maxh (s) \le 2 \kappa +2$. For $i \ge 2$, let $f \: \St_i/\St_i' \to V$ be a $\Gamma$--equivariant map, that we see as a map defined on $\St_i$ which is null on $\St_i'$. Then $f \circ \varphi_i \: C_i( \mathscr R_\kappa(X)) \to V$ is a bounded map: in fact,
$$ \sup \{ \|f \circ \varphi_i (s)\| \: s \textrm{ is an $i$--dimensional simplex in }  \mathscr R_\kappa(X) \} = $$ $$ = \sup \{ \|f \circ \varphi_i (s)\| \: \maxh(s) \le 2 \kappa +2\} < \infty, $$
because, up to the $\Gamma$--action, there is only a finite number of simplices $s$ with $\maxh(s) \le 2 \kappa$. Moreover, $f \circ \varphi_i \circ \psi_i$ is also bounded since, for every simplex $\Delta \in \st_{(k)}$,
$$ \|f \circ \varphi_i \circ \psi_i(\Delta)\| = \|f \circ \varphi_i \circ z_i(\Delta)\|, $$
and $z_i$ is a bounded map.

We now construct $\psi_*$, inductively verifying that it satisfies conditions (A), \ldots, (D) above. Recall that, by our hypotheses, $X$ is a subcomplex of $ \mathscr R_\kappa(X)$. Let $Q$ be the bicombing of Theorem \ref{main theorem GM}. Since $Q$ is quasi-geodesic and $C$--horoballs are convex, it follows that $Q(a, b)$ is completely contained in a $C$--horoball $ \mathscr H_C$ if $a$ and $b$ lie in $ \mathscr H_L$, for $L$ sufficiently large. Therefore for such an $L$ we put
$$ \psi_0(g, i) := (g, i, L) $$ 
$$ \psi_1((g, i), (h, j)) := Q(\psi_0(g, i), \psi_0(h, j)) \in C_1(X) \subset C_1( \mathscr R_\kappa(X)). $$

In order to simplify our notation, we denote by $\Delta^{i}$ a generic $i$--dimensional simplex in $\St$. If $\Delta^2 = (p_0, p_1, p_2)$, we write 
$$ \psi_1( \partial \Delta^2) = z(\Delta^2) + w(\Delta^2), $$
where $z(\Delta^2) := z(\psi_0(p_0), \psi_0(p_1), \psi_0(p_2))$ as in the notation of Theorem \ref{main theorem GM}, and $w(\Delta^2) = \psi_1( \partial \Delta^2) - z(\Delta^2)$. 

Notice that the cycles $z(\Delta^2)$ fulfill the conditions of Theorem \ref{cicli in quasi grafi} for a uniform constant $L$ and with $\maxh (z(\Delta^2))$ uniformly bounded. Therefore we can fill $z(\Delta^2)$ with a chain $z_2(\Delta^2)$, where $\maxh(z_2(\Delta^2))$ and its norm $\|z_2(\Delta^2)\|$ are uniformly bounded (i.e. independently of $\Delta^2$), and moreover $\supp (z_2(\Delta^2))$ is contained in some $C$--horoball, if the same is true for $\supp (z(\Delta^2))$.
We extend $z$ and $z_2$ by linearity. In what follows, all fillings are required to satisfy the conditions of Theorem \ref{cicli in quasi grafi}. We have 
$$ z(\partial \Delta^3) + w(\partial \Delta^3) = 0 $$
hence $-z(\partial \Delta^3) = w(\partial \Delta^3)$ is a $1$--dimensional cycle with bounded norm and minimun height at least $C$. Hence $\supp(w( \partial \Delta^3))$ is contained in the union of some $C$--horoballs. Since the $C$--horoballs of $X$ are disjoint complexes and because of (4) of Lemma \ref{lemma locale}, we have that $w( \partial \Delta^3) {}_{\big| \mathscr H_C}$ is a cycle for every $C$--horoball $\mathscr H_C$.

Let $\omega_2(\Delta^3)$ be a filling of $w( \partial \Delta^3)$ as in Theorem \ref{cicli in quasi grafi}, i.e.
$$ \partial \omega_2(\Delta^3) = -z(\partial \Delta^3) = w(\partial \Delta^3). $$ 
Applying Theorem \ref{cicli in quasi grafi} we find a chain $z_3$ such that:
$$ \partial (z_3(\Delta^3)) = z_2(\partial(\Delta^3)) + \omega_2(\Delta^3). $$
Fix $i \ge 4$, and suppose by induction that, for any $i$--symplex in $St_i$
$$ \partial (z_i(\Delta^i)) = z_{i-1}(\partial \Delta^i) + \omega_{i-1}(\Delta^i), $$
with $z_i$, $z_{i-1}$ and $\omega_{i-1}$ of uniformly bounded maximum height and $\ell^1$--norm, and such that $\minh(\omega_{i-1}) \ge C$. Moreover, suppose that the geometric conditions of Theorem \ref{cicli in quasi grafi} for $z_i$, $z_{i-1}$ and $\omega_{i-1}$ are also satisfied, where $n$ and $L$ in the statement of Theorem \ref{cicli in quasi grafi} that only depends on the dimension $i$. Then
$$ \partial (z_i( \partial \Delta^{i+1})) = z_{i-1}(\partial (\partial \Delta^{i+1})) + \omega_{i-1}(\Delta^{i+1}) = \omega_{i-1}(\partial\Delta^ {i+1}), $$
hence we can find a filling $\omega_{i}(\Delta^{i+1})$ of the cycle $-\omega_{i-1}(\partial \Delta^{i+1})$. Finally, we define $z_{i+1}$ in such a way that
$$ \partial (z_{i+1}(p_0, \ldots, p_{i+1})) = z_{i}(\partial(p_0, \ldots, p_{i+1})) + \omega_{i}(p_0, \ldots, p_{i+1}). $$

All inductive conditions are satisfied. 

Now we consider the construction of $w_*$. Similarly as before, by Theorem \ref{main theorem GM}, $\minh(w(\Delta^2)) \ge C$. Hence $w(\Delta^2) {}_{\big| \mathscr H_C}$ is a cycle for every $C$--horoball $\mathscr H_C$. By the contractibility of the $C$--horoballs (Corollary \ref{ohq3iulrejk}), we can fill every $w(\Delta^2) {}_{\big| \mathscr H_C}$ in $ \mathscr H_C$. Let $w_2(\Delta^2)$ be a filling of $w(\Delta^2)$ given by filling any $w(\Delta^2) {}_{\big| \mathscr H_C}$ in the same $C$--horoball. Note that we have defined $\omega_*$ in such a way that $ \partial \omega_2(\Delta^{3}) = w(\partial \Delta^{3})$, and $\partial \omega_{i+1}(\Delta^{i+2}) = -\omega_i(\partial \Delta^{i+2})$ for $i \ge 2$. We have that
$$ \partial w_2( \partial \Delta^3) = w( \partial \Delta^3) = \partial \omega_2( \Delta^3). $$
Hence we can define $w_3(\Delta^3)$ in such a way that
$$ \partial w_3(\Delta^3) = w_2( \partial \Delta^3) - \omega_2(\Delta^3). $$
Now, fix $i \ge 4$, and suppose by induction that 
$$ \partial w_i(\Delta^i) = w_{i-1}( \partial \Delta^i) - \omega_{i-1}(\Delta^i). $$
Then $\partial w_i( \partial \Delta^{i+1}) = - \omega_{i-1}( \partial \Delta^{i+1}) = \partial \omega_i (\Delta^{i+1})$, hence $w_i( \partial \Delta^{i+1}) - \omega_i (\Delta^{i+1})$ is a cycle, which we can fill by $w_{i+1}(\Delta^{i+1})$.

This concludes the construction of $\psi_*$, whence the proof of Theorem \ref{main MIO theorem} (a).

\section{\label{sezione 7} Applications}

Let $(X, A)$ be a topological pair. Let $S_*(X)$ be the singular complex of $X$ with real coefficients. In other words, $S_k(X)$ is the real vector space whose basis is the set $C^0(\Delta^k, X)$ of singular $k$--dimensional simplices in $X$, and we take the usual boundary operator $ \partial _k \: S_k(X) \to S_{k-1}(X)$, for $k \ge 1$. The natural inclusion of complexes $S_*(A) \hookrightarrow S_*(X)$ allows us to define the relative singular complex $S_*(X, A) := S_*(X)/S_*(A)$. Dually, we define the relative singular cocomplex as
$$ S^*(X, A) = \hom(S_*(X, A), \R), $$
where $\hom(S_*(X, A), \R)$ denotes the set of real linear maps on $S_*(X, A)$. We put $S^*(X, \emptyset) =: S^*(X)$. We will often identify $S^*(X, A)$ with the subspace of $S^*(X)$ whose elements are null on $S_*(A)$.  We put an $\ell^1$--norm on $S_*(X, A)$ through the identification: 
$$S_*(X, A) \sim {\R(C^0(\Delta^i, X) \setminus C^0(\Delta^i, A))}.$$
Given a cochain $f \in S^*(X, A)$, the (possibly infinite) $\ell^\infty$--norm of $f$ is
$$ \|f\|_\infty := \sup\{|f(c)| \: c \in S_*(X, A),\, \|c\| \le 1\}. $$ 

We denote by $S_b^*(X, A)$ the subcocomplex of $S^*(X, A)$ whose elements have finite $\ell^\infty$--norm. Since the boundary operator $ \partial_* \: S_*(X, A) \to S_{*-1}(X, A) $ is bounded with respect to the $\ell^1$--norms, its dual maps bounded cochains into bounded cochains (and is bounded with respect to the $\ell^\infty$--norm). Therefore $S_b^*(X, A)$ is indeed a cocomplex.

The following definition appeared for the first time in \cite[Section 4.1]{VaBC}.

\begin{dfn} Given a topological pair $(X, A)$, the \textbf{relative bounded cohomology $H_b^*(X, A)$} is the cohomology of the cocomplex $S_b^*(X, A)$.
\end{dfn}

\begin{dfn} \label{ } Let $S_*(X, A)$ be the real singular chain complex of a topological pair. The norm on $S_*(X, A)$ descends to a natural semi-norm on homology, called \textbf{Gromov norm}: for every $\alpha \in H_*(X, A)$,
$$ \|\alpha\| = \inf \left\{\|c\| \: c \in S_*(X, A), \, [c] = \alpha  \right\}. $$
If $M$ is an $n$--dimensional oriented compact manifold with boundary, the \textbf{simplicial volume} of $M$ is the Gromov norm of the fundamental class in $H_n(M, \partial M)$.
\end{dfn}

\begin{dfn} A topological pair $(X, Y)$ is a \textbf{classifying space} for the group-pair $(\Gamma, \{\Gamma_i\}_{i \in I})$ if \begin{enumerate}
    \item $X$ is path-connected, and $Y = \bigsqcup_{i \in I} Y_i$ is a disjoint union of path-connected subspaces $Y_i$ of $X$ parametrized by $I$;
    \item there are basepoints $x \in X$ and $y_i \in Y_i$, and isomorphisms $\pi_1(X, x) \sim \Gamma$ and $\pi_1(Y_i, y_i) \sim \Gamma_i$;
    \item the $Y_i$ are $\pi_1$--injective in $X$, and there are paths $\gamma_i$ from $x$ to $y_i$ such that the induced injections
$$ \pi_1(Y_i, y_i) \hookrightarrow \pi_1(X, x) $$ 
correspond to the inclusions $\Gamma_i \hookrightarrow \Gamma$ under the isomorphisms above;
	\item $X$ and $Y$ are aspherical.
\end{enumerate}
\end{dfn}

The following theorem applies in particular to negatively curved compact manifolds with totally geodesic boundary.

\begin{theorem} \label{oq3yhilruedj} Let $(X, Y)$ be a classifying space of a relatively hyperbolic pair $(\Gamma, \Gamma')$. Then the Gromov norm on $H_k(X, Y)$ is a norm for any $k \ge 2$.
\end{theorem}

\begin{proof} Let $H^*(\Gamma, \Gamma')$ be the relative cohomology of $(\Gamma, \Gamma')$ as defined in \cite{BE} (the definition of Bieri and Eckmann is completely analogous to the one of Mineyev and Yaman, but without any reference on the norm). It is possible to define natural maps
\begin{equation} \label{3oqregwlj} H_b^*(X, Y) \to H_b^*(\Gamma, \Gamma') \to H^*(\Gamma, \Gamma') \to H^*(X, Y) \end{equation}
such that the first map is an isometric isomorphism, the third one is an isomorphism, and the compositions of all maps in \eqref{3oqregwlj} is the comparison map from singular bounded cohomology to singular cohomology (the fact that the first map is an isometry also follows from weaker hypotheses: see \cite[Theorem 5.3.11]{Blank}).  By hypothesis, the second map in \eqref{3oqregwlj} is surjective. Hence the conclusion follows from the following proposition (\cite[Proposition 54]{MY}), which is the relative version of an observation by Gromov (\cite[p. 17]{VaBC}) and could be generalized for any normed chain complex (see \cite[Theorem 3.8]{LohPhD}).

\begin{prop} \label{ } For any $z \in H_k(Y, Y'; \R)$, 
$$ \|z\| = \sup \left( \Big\{ \frac{1}{\|\beta\|_\infty} \: \beta \in H_b^k(Y,Y'; \R) \: \left\langle \beta, z \right\rangle = 1 \Big\} \cup \{0\} \right). $$
\end{prop}
\end{proof}

Now we consider our second application: a relatively hyperbolic group-pair has finite cohomological dimension. More precisely

\begin{theorem} \label{31fvqrwac} Let $(\Gamma, \Gamma')$ be a relatively hyperbolic pair. Then there is $n \in \N$ such that, for every $m > n$ and every bounded $\Gamma$--module $V$, $H^{m}(\Gamma, \Gamma'; V) = 0$.
\end{theorem}

We note that this theorem admits a straightforward proof in the case of a torsion-free hyperbolic group $\Gamma$. Indeed, consider a contractible Rips complex $Y$ over the Cayley-graph $X$ of $\Gamma$. The complex $Y$ is finite dimensional by the uniform local compactness of $X$. Since $Y$ is contractible and $\Gamma$ acts freely on it, the cohomology of $\Gamma$ is isomorphic to the (simplicial) cohomology of $Y$, whence the conclusion.

Let $\mathscr R_\kappa := \mathscr R_\kappa(X)$ be the Rips complex associated to a cusped space $X$ of the relatively hyperbolic pair $(\Gamma, \Gamma')$, as described in Corollary \ref{ohq3iulrejk}. Then

\begin{lem} \label{3qerqre} For every $C > 0$ there exists $n \in \N$ such that, for every $m \ge n$ and for every $m$--simplex $\Delta$ of $ \mathscr R_\kappa(X)$, we have $\minh (\Delta) > C$.
\end{lem}

\begin{proof} For $m$ sufficiently large, every subset $A \subseteq \mathscr R_\kappa(X)^{(0)}$ of cardinality $m$ and such that $\minh (A) \le C$ has $X$--diameter greater than $\kappa$. This follows easily from the fact that, up to $\Gamma$-action, there are only finitely many such sets $A$. Therefore, by definition of Rips complex, the conclusion follows.
\end{proof}

We can now prove Theorem \ref{31fvqrwac}.

\begin{proof} Let $V$ be a bounded $\Gamma$--module, and let $f$ be a cochain in $\hom^\Gamma(\St_k^{\rel}, V)$, which we can see as a $\Gamma$--equivariant map which is null on $\St_k' \subset \St_k$. By Lemma \ref{3qerqre}, $f \circ \varphi_k \circ \psi_k = 0$, for $k$ sufficiently big and independent of $f$. If $f$ is a cocycle, by Lemma \ref{il lemma omologico o3827y}, $f$ is cohomologous to the null map, hence $H^k(\Gamma, \Gamma'; V) = 0$ by the arbitrariness of $f$.
\end{proof}

\section{\label{sezione 8} Proof of part \textrm{$(b)$} of Theorem \ref{main MIO theorem}}

In the following we will work in the category of \emph{combinatorial cell complexes} (see \cite[8A.1]{BriHae}). We are particularly interested in the $2$--skeleton of a combinatorial complex $X$. This is described as follows: $X^{(1)}$ is any graph, and the $2$--cells are $l$--polygons $e_\lambda$, $l\ge2$, such that the attaching map $ \partial e_\lambda \to X^{(1)}$ is a loop whose restriction to any open cell of $\partial e_\lambda$ (i.e.: open edge or point) is a homeomorphism to some open cell of $X^{(1)}$.

The following characterization of relative hyperbolicity was proved by Bowditch in \cite[Definition 2]{BOW}.

\begin{dfn} \label{qo3hiulendjsl} Let $G$ be a graph. A circuit in $G$ is a closed path that meets any vertex at most once. We say that $G$ is \textbf{fine} if, for any edge, the set of circuits which contain $e$ is finite. A group $\Gamma$ is \textbf{hyperbolic relative to a finite collection of subgroups $\Gamma'$} if $\Gamma$ acts on a connected, fine, $\delta$--hyperbolic graph $G$ with finite edge stabilizers, finitely many orbits of edges, and $\Gamma'$ is a set of representatives of distinct conjugacy classes of vertex stabilizers (such that each infinite stabilizer is represented).
\end{dfn}

\begin{dfn} \cite[Definition 1.2]{pedro}\label{ } Let $\K \in \{\Z, \Q, \R\}$ and let $X$ be a combinatorial cell complex. \textbf{The homological Dehn function of $X$ over} $\K$ is the map $FV_{X, \K} \: \N \to \R$ defined by
$$ FV_{X, \K}(k) := \sup\{\|\gamma\|_{f, \K} \: \gamma \in Z_1(X, \Z),\, \|\gamma\| \le k\} $$
where: 
$$ \|\gamma\|_{f, \K} := \inf \{\|\mu\| \: \mu \in C_2(X, \K), \partial \mu = \gamma\}. $$
\end{dfn}

By a result given in \cite{Min2} (which generalizes \cite[Theorem 3.3]{AG}) the linearity of $FV_{X, \K}$ is equivalent to the undistortedness of the boundary $\partial _2 \: C_2(X, \K) \to C_1(X, \K)$, if $\K \in \left\{\Q, \R\right\}$. Indeed, given a cycle $z \in Z_1(X, \K)$, we can express it as a sum of circuits $z = \sum_{c} a_c c$ in such a way that $\K \ni a_c \ge 0$ for all $c$, and $\|z\| = \sum_c a_c \|c\|$ (see \cite[Theorem 6 (b)]{Min2}, with $T = \emptyset$). Suppose that $\|c\|_{f, \K} \le K\|c\|$ for some constant $K \ge 0$ and any circuit $c$. Then
$$ \|z\|_{f, \K} \le \sum_{c} g(c) \|c\|_{f, \K} \le K \sum_{c} g(c) \|c\| = K \|z\|. $$ 
Moreover, we have

\begin{prop} \label{o847ywteiuv} Let $X$ be a simply connected combinatorial cell complex. Then 
$$ FV_{X, \Q} = FV_{X, \R}. $$
\end{prop}

\begin{proof} We have to prove that $FV_{X, \Q} \le FV_{X, \R}$, since the opposite inequality is clear. Let $\gamma \in Z_1(X, \Z)$, and let $a = \sum_i \lambda_i \sigma_i \in C_2(X, \R)$ be such that $ \partial a = \gamma$. We approximate the $\lambda_i$ with rational coefficients $\lambda_i'$, in such a way that, if $a' = \sum_i \lambda_i' \sigma_i$, then $\|\partial (a - a')\| \le \varepsilon$. Let $W$ be the normed subspace of $Z_1(X, \Q)$ whose elements are $\Q$--linear combinations of faces of the $\sigma_i$. Let $\theta \: W \to C_2(X, \Q)$ be a $\Q$--linear map such that $ \partial \theta (w) = w$ for all $w \in W$. Since $W$ is finite-dimensional, $\theta$ is bounded. Moreover, $\partial (a-a') \in W$. Hence
$$ \partial (a' + \theta \partial (a-a')) = \gamma $$ 
and
$$ \|a' + \theta \partial (a-a') \| \le \|a'\| + \|\theta\|_\infty \varepsilon $$
from which the conclusion follows immediately by the arbitrariness of $ \varepsilon$.
\end{proof}

The following lemma is stated as such in \cite{pedro}, but is proven in \cite[Theorem 2.30]{GM} with a different notation.

\begin{lem} \cite[Theorem 3.4]{pedro} \label{pedro 3.4} Let $X$ be a simply connected complex such that there is a bound on the length of attaching maps of $2$--cells. If $FV_{X,\Q}$ is bounded by a linear function, then the $1$--skeleton of $X$ is a hyperbolic graph.
\end{lem}

The following theorem is a slight modification of the ``if part'' of \cite[Theorem 1.8]{pedro}: we require the complex to be simply connected instead of $1$--acyclic, and we write $FV_{X, \Q}$ instead of $FV_{X, \Z}$ in (2).

\begin{theorem} \label{oq3847hfiuwel} Let $(\Gamma, \Gamma')$ be a group-pair. Then $(\Gamma, \Gamma')$ is hyperbolic if there is a simply connected combinatorial complex $X$ such that \begin{enumerate}
    \item $\Gamma$ acts cocompactly on $X^{(2)}$;
    \item $FV_{X, \Q}(k) \le Ck$ for every $k \in \N$;
    \item the stabilizers in $\Gamma$ of edges are finite;
    \item $\Gamma'$ is a set of representatives of (distinct) conjugacy classes of stabilizers of $0$--cells such that each infinite stabilizer is represented. This means that there is an injection
$$ \Gamma' \to \{[\stab(v)] \: v \in X^{(0)} \} \qquad \Gamma_i \mapsto [\Gamma_i] $$
(where $[H]$ denotes the conjugacy class of a subgroup $H$ of $\Gamma$) whose image contains all conjugacy classes of infinite stabilizers in $\Gamma$ of vertices in $X^{(0)}$.
\end{enumerate}
\end{theorem}

\begin{proof} Points (1) and (2) imply the hyperbolicity of the graph by Lemma \ref{pedro 3.4}. Hence, in order to apply Bowditch's characterization of relative hyperbolicity 
 it remains to prove that $X^{(1)}$ is fine.

Condition (1) in the statement implies that there is a bound on the number of edges on the boundary of $2$--cells. Moreover, conditions (1) and (3) imply that any edge belongs to just a finite number of $2$--cells (because edge-stabilizers act cocompactly on the $2$--cells adjacent to the edge). 

We conclude by mean of the following lemma, which is proven in \cite[Theorem 1.6 (2)]{pedro}.

\begin{lem} \label{owhreflm} Let $X$ be a simply connected combinatorial cell complex such that each $1$--cell is adjacent to finitely many $2$--cells and there is a bound on the length of attaching maps of $2$--cells. Suppose that there is $C \ge 0$ such that 
$$ FV_{X, \Q}(k) \le Ck  \qquad \forall k \in \N. $$
Then $X^{(1)}$ is fine.
\end{lem}
\end{proof}

\begin{dfn} \label{q3o47bfihuadjks} \cite[Definition 2.1]{Osin} \label{ } Let $(\Gamma, \Gamma')$ be a group-pair. We say that $\Gamma$ is \textbf{finitely presented relative to $\Gamma'$} if: \begin{enumerate}
    \item $\Gamma$ is generated by $\bigcup_{i \in I} \Gamma_i$ and a finite subset $ \mathscr A$ of $\Gamma$;
    \item the kernel of the natural projection
    $$ F( \mathscr A) * (*_{i \in I} \Gamma_i) \to \Gamma $$
is generated -- as a normal subgroup of $F( \mathscr A) * (*_{i \in I} \Gamma_i)$ -- by a finite set $ \mathscr R \subseteq F( \mathscr A) * (*_{i \in I} \Gamma_i)$ of \textbf{relations}.
\end{enumerate}
In this case, the datum of $ \left\langle \mathscr A, \Gamma'| \mathscr R \right\rangle$ is a \textbf{finite presentation} of $(\Gamma, \Gamma')$.
\end{dfn}

\begin{ntz} \label{3o4iuhqfew} From now on, we will assume that $(\Gamma, \Gamma' = \left\{\Gamma_i\right\}_{i \in I = \left\{1, \ldots, n\right\}})$ is a finitely presented group-pair, and that $\Gamma$ is finitely generated. By a result in \cite[Proposition 2.29]{Osin}, it follows that the groups in $\Gamma'$ are finitely generated too. 
\end{ntz}

Since there exist slightly different definitions of Cayley-graph in the literature, from now on we will rely on the following one. Let $S$ be a (non-necessarily symmetric) generating set of a group $\Gamma$. The \textbf{Cayley graph $G = G(\Gamma, S)$ of $\Gamma$ w.r.t. $S$} is the graph whose $0$--skeleton is $\Gamma$ and with an edge connecting $x$ and $xs$ labelled by $(x, s)$, for any $(x, s) \in \Gamma \times S$.

Notice that $\Gamma$ acts freely and isometrically on $G(\Gamma, S)$ by mapping the vertex $x$ to $\gamma x$ and the edge $(x, s)$ to the edge $(\gamma x, s)$.

\begin{dfn} \cite[Relative Cayley complex]{GM} Let $G := G(\Gamma, S)$ be the Cayley graph of $\Gamma$, with respect to some compatible generating set $S$. Consider the graph $G^I$ constructed as follows: \begin{enumerate} 
    \item $(G^I)^{(0)} = G^{(0)} \times I$; 
    \item for any $i \in I$, $G^i := G \times \left\{i\right\}$.
For all $v \in G^{(0)}$ and $i \ne j$ there is a single edge connecting $(v, i)$ and $(v, j)$.
\end{enumerate}

We call the edges contained in some $G^i$ \textbf{horizontal}, and the other ones \textbf{vertical}.

By writing elements of $ \mathscr R$ with the alphabet $S$, we can, for every $i$, (non-uniquely) associate them to loops in $G^i$ based in $1$. We add $\Gamma$--equivariantly $2$--cells to those loops and their $\Gamma$--translates. Let $i \ne j \in I$. If $e^i$ is an edge in $G^i$ we have a corresponding edge $e^{j}$ in $G^{j}$, and two vertical edges connecting the initial and final points of $e^i$ and $e^{j}$. We add a rectangular $2$--cell to this quadrilateral. We denote by $Cay(\Gamma, \Gamma')$ the $2$--dimensional combinatorial cell complex obtained in this way, and call it the \textbf{relative Cayley-complex of $(\Gamma, \Gamma')$ (w.r.t. $S$)}.
\end{dfn}

The group $\Gamma$ naturally acts on $Cay(\Gamma, \Gamma')$. 

\begin{dfn} \label{ } The \textbf{$2$--dimensional quotient complex $\widehat{ X} = \widehat{ X}(\Gamma, \Gamma')$} is the CW-complex obtained by collapsing to points the full subcomplexes of $Cay(\Gamma, \Gamma')$ whose vertices are contained in the same left coset of $\Gamma_i \times \left\{i\right\}$, $i \in I$.

\end{dfn}

\begin{rem} \label{own478tveam} This means that, if $Y_i$ is the full subcomplex of $Cay(\Gamma, \Gamma')$ whose vertices correspond to $\Gamma_i \times \left\{i\right\}$, then all (left) $\Gamma$-translates of $Y_i$ are collapsed to points. It is easily seen that $\widehat{X}$ could be given the structure of a combinatorial complex. 

At the $0$--dimensional level, we have a natural $\Gamma$--isomorphism $\Gamma / \Gamma' := \bigsqcup_{i \in I} \Gamma / \Gamma_i \to \widehat{ X}^{(0)}$. We use it to label the vertices of $\widehat{ X}^{(0)}$ by $\Gamma / \Gamma'$. Given an horizontal edge $(x, s)$ in $G^i \subseteq Cay(\Gamma, \Gamma')$, this is either collapsed to a point in $\widehat{ X}$ if $s \in \Gamma_i$, or is left unchanged. Hence the horizontal edges of $\widehat{ X}$ are naturally labelled by the set $\bigsqcup_{i \in I} \Gamma \times (S \setminus \Gamma_i)$. Notice that vertical edges are never collapsed. 

The complex $\widehat{X}$ carries a natural $\Gamma$--action. The action on the $0$--skeleton has already been described. A cell of dimension at least $1$ in $\widehat{ X}$ corresponds to exactly one cell of the same dimension in $Cay(\Gamma, \Gamma')$, hence the action of $\Gamma$ on $ \widehat X$ is defined accordingly. Notice that, since the action of $\Gamma$ on the Cayley complex is free, the same is true for the action of $\Gamma$ on the $1$--skeleton of $\widehat X$. In particular Condition (3) of Theorem \ref{oq3847hfiuwel} holds.
\end{rem}

\begin{prop} \label{} $\widehat{X}$ is simply connected.
\end{prop}

\begin{proof} Let $Y_i$ be the full subcomplex of $Cay(\Gamma, \Gamma')$ whose vertices are labelled by $\Gamma_i \times \left\{i\right\}$. For all $i \in I$, we add $\Gamma$--equivariantly $2$--cells to $Y_i$ and to its $\Gamma$--translates, in order to obtain a simply connected combinatorial complex $Z$.

The complex obtained by collapsing to points the full subcomplex of $Z$ containing $Y_i$, $i \in I$, and its $\Gamma$--translates is simply connected by an easy application of van-Kampen Theorem. Moreover, it is obviously homeomorphic to $\widehat{ X}$.
\end{proof}

We add $\Gamma$--equivariantly higher dimensional cells to $\widehat{ X}$ in order to make it a contractible combinatorial complex, that we also denote by $\widehat{ X}$, and call it the \textbf{quotient complex}. Consider the exact cellular sequence
$$ \cdots \to C_1(\widehat{ X}) \to C_0(\widehat{ X}) \to \R \to 0. $$
Recall that we have a $\Gamma$--isomorphism between the $\Gamma$--sets ${\widehat{ X}^{(0)}}$ and $\Gamma/\Gamma'$. Therefore, if $\Delta$ is the kernel of the augmentation map $\R(\Gamma / \Gamma') \to \R$ we also have the exact sequence
\begin{equation} \label{o3qhuilwjn} C_*(\widehat{ X}) \to \Delta \to 0. \end{equation}

By the following lemma, the sequence \eqref{o3qhuilwjn} provides a $\Gamma$--projective resolution of $\Delta$ (i.e., all the $\Gamma$--modules except $\Delta$ are $\Gamma$--projective).

\begin{lem} \label{} Let $X$ be a contractible CW-complex, and let $\Gamma$ act on $X$ through cellular homeomorphisms. Suppose that the stabilizers in $\Gamma$ of $1$--cells are finite. Then $C_k(X)$ is a $\Gamma$--projective module for every $k \ge 1$.
\end{lem}

\begin{proof} For $k \ge 1$, the stabilizer of any $k$--dimensional cell is finite. By (arbitrarily) choosing an orientation for every $k$--cell of $X$, we get a basis of $C_k(X)$. Then we conclude by applying Lemma \ref{Lemma 52 MY} to such a basis.
\end{proof}

We are now ready to prove Theorem \ref{main MIO theorem} (b). Most of the proof follows almost verbatim \cite[Theorem 57]{MY}. We note however that the existence of a combinatorial isoperimetric function required in the statement of Theorem 57 is never actually exploited in its proof.

\begin{proof} We will prove that $\widehat{ X}$ satisfies conditions (1), \ldots, (4) of Theorem \ref{oq3847hfiuwel}. Condition (4) is obvious. Since $\widehat{ X}^{(2)}$ is a quotient of the relative Cayley complex, the $\Gamma$--action on it is obviously cocompact, whence (1). Condition (3) was already proved in Remark \ref{own478tveam}.

Now, let $V := (B_1(\widehat{ X}); \|\cdot\|_f)$, where $B_1(\widehat{ X}) \subset C_1(\widehat{ X})$ is the set of boundaries and $\|\cdot\|_f$ is the \textbf{filling norm}
$$ \|c\|_f := \inf\{\|a\| \: a \in C_2(\widehat{ X}),\,\partial a = c\}. $$
This is actually a norm (and not just a semi-norm) because, by Condition (1), the boundary map $ \partial_2 \: C_2(\widehat{ X}) \to C_1(\widehat{ X})$ is bounded, with respect to the $\ell^1$--norms.

We have already seen that $\St_*^{\rel}(\Gamma, \Gamma')$ and $C_*(\widehat{ X})$ provide $\Gamma$--projective resolutions of $\Delta$. Hence by Lemma \ref{q3onvp31t4y2} there are, up to (non-bounded) $\Gamma$--homotopy, unique chain maps 
$$ \varphi_* \: \St_*^{\rel}(\Gamma, \Gamma') \to C_*(\widehat{ X}) \qquad \psi_* \: C_*(\widehat{ X}) \to \St_*^{\rel}(\Gamma, \Gamma') $$ that extend the identity on $\Delta$. Put
$$ u := \partial _2 \: C_2(\widehat{ X}) \to V. $$
The cochain $u$ is a cocycle. Since $\psi^2 \circ \varphi^2$ induces the identity in ordinary cohomology, there is $v \in C^1(X, V)$ such that
$$ u = \psi^2(\varphi^2(u)) + \delta v. $$
From the surjectivity hypothesis we get
$$ \varphi^2(u) = u' + \delta v', $$ for some bounded cocycle $u' \in \St^2_{\rel}(\Gamma, \Gamma'; V)$ and $v' \in C^1_{\rel}(\Gamma, \Gamma')$.
Let $b \in C_1(\widehat X)$ be a cycle, and let $a \in C_2(\widehat X)$ be a filling of $b$. Then
\begin{equation} \label{p4wtveafms}  b = \partial  a = \left\langle u,  a \right\rangle = \left\langle (\psi^2 \circ \varphi^2)(u) + \delta v,  a  \right\rangle = \left\langle  (\psi^2 \circ \varphi^2)(u),  a \right\rangle + \left\langle v,  b \right\rangle. \end{equation}
By Corollary \ref{Equation (29), p. 38} we have
$$ \left\langle (\psi^2 \circ \varphi^2)(u), a \right\rangle = \left\langle  \varphi^2(u), \psi_2(a) \right\rangle = \left\langle  \varphi^2(u), [y, \partial (\psi_2(a))]_{\rel} \right\rangle = \left\langle  \varphi^2(u), [y, \psi_1(b)]_{\rel} \right\rangle = $$ $$ = \left\langle  u' + \delta v', [y, \psi_1(b)]_{\rel}  \right\rangle = \left\langle  u', [y, \psi_1(b)]_{\rel} \right\rangle + \left\langle  v', \partial [y, \psi_1(b)]_{\rel} \right\rangle = $$ $$ = \left\langle  u', [y, \psi_1(b)]_{\rel} \right\rangle + \left\langle  v', \psi_1(b) \right\rangle = \left\langle  u', [y, \psi_1(b)]_{\rel} \right\rangle + \left\langle  \psi^1(v'), b \right\rangle. $$
Summarizing,
$$ b = \left\langle  u', [y, \psi_1(b)]_{\rel} \right\rangle + \left\langle  \psi^1(v') + v, b \right\rangle. $$
Hence
$$ |b|_f \le \Big| \left\langle  u', [y, \psi_1(b)]_{\rel} \right\rangle \Big|_f + \Big| \left\langle  \psi^1(v') + v, b \right\rangle \Big|_f \le |u'|_\infty \Big\| [y, \psi_1(b)]_{\rel} \Big\| + \Big| \psi^1(v') + v \Big|_\infty \|b\| \le $$ $$ \le 3 |u'|_\infty \|\psi_1(b)\| + \Big| \psi^1(v') + v \Big|_\infty \|b\| \le \left( 3|u'|_\infty |\psi_1|_\infty  + \Big| \psi^1(v') + v \Big|_{\infty} \right) \|b\|. $$

Hence it remains to prove that $\left( 3|u'|_\infty |\psi_1|_\infty + \Big| \psi^1(v') + v \Big|_{\infty} \right)$ is bounded. 

The cocycle $u'$ is bounded by definition. Moreover $\psi_1 \: C_1(\widehat{X}) \to \St_1^{\rel}(\Gamma, \Gamma')$ and $\psi^1(v') + v \: C_1(\widehat{X}) \to V$ are $\Gamma$--equivariant, hence also bounded by the cocompactness of the action of $\Gamma$ over ${\widehat{ X}}^{(2)}$.

It follows that $ \partial \: C_2(X, \R) \to C_1(X, \R)$ is undistorted, hence $FV_{X, \R} = FV_{X, \Q}$ is linearly bounded.
\end{proof}

\begin{rem} \label{} The proof of part (b) of Theorem \ref{main MIO theorem} could be adapted, as in \cite{MY}, by weakening the hypotheses in the statement by requiring the surjectivity only for \emph{Banach} coefficients (that is: Banach spaces equipped with an isometric $\Gamma$--action).
\end{rem}

\subsection{\label{addendum i}Addendum I: the relative cone of \cite{MY}}

We recall the definition and properties of the \emph{relative cone} given in \cite[10.1,\ldots, 10.5]{MY}. 

Consider the (non-linear) map: $ \Phi \: \St_0(\Gamma) \to \St_1(\Gamma)$ 
$$ \Phi(c) := \frac{1}{\sum_{x \in I\Gamma} \alpha_x^+} \sum_{x, y \in I\Gamma} \alpha_x^- \alpha_y^+ [x, y], $$
where $c \in \St_0$ is written as: $c = \sum_{x} \alpha_x^+ x - \sum_{x} \alpha_x^-x$ with all the $\alpha_x^+$ and $\alpha_x^-$ non-negative and, for any $x \in I\Gamma$, $\alpha_x^+ = 0$ or $\alpha_x^- = 0$. The following fact is immediate.

\begin{prop} \label{ } For every $c \in \St_0$, $\|\Phi(c)\| \le \|c\|$. If $c$ is contained in the kernel of the map $\St_0 \to \R$, then
\begin{equation} \label{puh43qiwenjkl} \partial \Phi(c) = c. \end{equation}
\end{prop}

\begin{dfn} \cite[The absolute cone]{MY} Fix $y \in I\Gamma$ and $k \ge 0$. The \textbf{$k$--dimensional cone} (associated to $y$) is the map $[y, \,\cdot] \: \St_k \to \St_{k+1}$ given by
$$ [y, (\gamma_0, \ldots, \gamma_k)] := (y, \gamma_0, \ldots, \gamma_k) \qquad \forall \gamma_0, \ldots, \gamma_k \in I\Gamma $$ 
and extended over the whole $\St_k$ by linearity.
\end{dfn}

It is trivialy seen that $[y, \cdot]$ is a linear map of norm $1$ for every $k$. Moreover, 
\begin{equation} \label{34oqfuihewnjd} \partial [y, z] = z, \end{equation}
for any $k$--dimensional cycle $z$, $k \ge 0$.

Let $pr_* \: \St_* \to \St_*^{\rel}$ be the projection, and let $j_* \: \St_*^{\rel} \to \St_*$ be the obvious right inverse of $pr_*$. This map has norm $1$. For any left coset $s \in \Gamma / \Gamma'$ and $a \in \St_*$, let $ \partial^s(a)$ be the restriction of $\partial a$ to $s$.

\begin{dfn} \cite[The relative cone]{MY} Fix $y \in I\Gamma$. The \textbf{1--dimensional relative cone} (associated to $y$) is the (non-linear!) map
$$ [y, \,\cdot]_{\rel} \: \St_1^{\rel} \to \St_2^{\rel} \qquad  [y, b]_{\rel} := pr_2 \left[y, j(b) - \sum_{s \in \Gamma / \Gamma'} \Phi[ \partial ^s(j(b))]\right] \qquad \forall b \in \St_1^{\rel}(\Gamma). $$
\end{dfn}

We prove Proposition \ref{The relative cone}. Let $b \in \st_1^\rel$. 
$$ \left\| [y, b]_{\rel} \right\| = \left\| pr \left[y, j(b) - \sum_{s \in \Gamma / \Gamma'} \Phi[ \partial ^s(j(b))]\right] \right\| \le \left\| \left[y, j(b) - \sum_{s \in \Gamma / \Gamma'} \Phi[ \partial ^s(j(b))]\right] \right\| = $$ $$ = \left\| \left[j(b)  - \sum_{s \in \Gamma / \Gamma'} \Phi[ \partial ^s(j(b))]\right] \right\| \le \|j(b)\| + \sum_{s \in \Gamma / \Gamma'} \left\|  \Phi[ \partial ^s(j(b))] \right\| \le \|b\| + \sum_{s \in \Gamma / \Gamma'} \left\| \partial^s(j(b))] \right\| \le  $$ $$ \le \|b\| + \left\| \sum_{s \in \Gamma / \Gamma'}  \partial^s(j(b))] \right\| = \|b\| + \left\| \partial (j(b))] \right\| \le \|b\| + 2 \| j(b)\| \le \|b\| + 2 \|b\| = 3 \|b\|. $$
Now, let $b \in \St_1^{\rel}$ be a cycle with respect to the augmentation map
$$ \St_1^{\rel} \to \Delta \qquad pr_2 (x, y) \mapsto [y] - [x] $$
(where $[\cdot]$ refers to the class in $\Gamma / \Gamma'$). We prove that
\begin{equation} \label{4p58gijkop} \partial ^{\rel}[y, b]_{\rel} = b. \end{equation}
Write
$$ b =  \sum_i \lambda_i [x_i, y_i] $$
(if $[x_i, y_i] \notin \Gamma'$ we will identify $[x_i, y_i]$ and $pr([x_i, y_i])$). By hypothesis: 
$$ \sum_i \lambda_i ([y_i] - [x_i]) = 0 \in \Delta \subseteq \R(\Gamma / \Gamma'). $$
Equivalently, for any $s \in \Gamma/\Gamma'$: $\sum_{y_i \in s} \lambda_i - \sum_{x_j \in s} \lambda_j = 0$. Hence: $\partial^s j(b) = \sum_{y_i \in s} \lambda_i y_i - \sum_{x_j \in s} \lambda_j x_j$ is a cycle with respect to the augmentation map: $\St_0(\Gamma, \Gamma') \to \R \to 0$. Therefore we get
$$ \partial \sum_{s \in \Gamma / \Gamma'} \Phi\left( \partial ^s(j(b))\right) = \sum_{s \in \Gamma / \Gamma'} \partial \Phi \left( \partial ^s(j(b))\right) = \sum_{s \in \Gamma / \Gamma'} \partial ^s(j(b)), $$ because of \ref{puh43qiwenjkl}. Moreover
$$ \partial ^{\rel} [y, b]_{\rel} := \partial^{\rel} pr_2 \left[y, j(b) - \sum_{s \in \Gamma / \Gamma'} \Phi[ \partial ^s(j(b))]\right] = pr_2 \partial \left[y, j(b) - \sum_{s \in \Gamma / \Gamma'} \Phi[ \partial ^s(j(b))]\right] = $$ 
$$= pr_2 \left(j(b) - \sum_{s \in \Gamma / \Gamma'} \Phi[ \partial ^s(j(b))]\right) = b  $$
because $\Phi[ \partial ^s(j(b))] \in \St_2'$ for every $s \in \Gamma/\Gamma'$.

\subsection{\label{addendum ii}Addendum II: coincidence between  Mineyev-Yaman and Blank definitions of relative bounded cohomology}

We prove that Blank's definition of relative bounded cohomology for pair of groupoids, when restricted to group-pairs, coincides with the one of Mineyev and Yaman, up to isometry. 

First we briefly sketch Blank's definition of relative bounded cohomology for groupoids. For more details, see \cite[Chapter 3]{Blank}. If $G$ is a groupoid, we write ``$g \in G$'' if $g \in \morf(e, f)$, i.e. if $g$ is a morphism between two objects $e$ and $f$ of $G$. In that case we also put $s(g) = e$, $t(g) = b$. A \textbf{bounded $G$--module $V$} is a set of normed vector spaces = $\{V_e\}_{e \in \obj(G)}$ which carries a bounded groupoid $G$--action. This means that to any $g \in G$ an operator $\rho_g \: V_{s(e)} \to V_{t(e)}$ is assigned whose norm is bounded independently of $g \in G$, and the composition rule: $\rho_{g \circ h} = \rho_g \circ \rho_h$ is respected when defined (our definition of bounded groupoid module is slightly more general than that of normed $G$--module in \cite[Chapter 3.3.1]{Blank} in that we consider actions by uniformly bounded operators on normed spaces, instead of isometries on Banach spaces). If $V$ and $W$ are bounded $G$--modules, by $\hom_G^b(V, W)$ we mean the space of bounded maps $(f_e \: V_e \to W_e)_{e \in \obj(G)}$ such that $\rho_g \circ f_{s(g)} = f_{t(g)} \circ \rho_g$ and $\|f_e\| \le L$ for some $L$ independent of $e \in \obj(G)$.

To $G$ we associate the \textbf{Bar resolution} $\{\mathscr C_n(G)\}_{n \in \N}$ defined as follows. For $n \in \N$ put $\mathscr C_n(G) := \{\mathscr C_n(G)\}_{e \in \obj (G)}$, where $(\mathscr C_k(G))_e$ is the normed space generated by the $n+1$--tuples $(g_0, \ldots, g_n)$ such that $s(g_0) = e$ and $s(g_j) = t(g_{j-1})$, for $1 \le j \le n$, with the corresponding $\ell^1$--norm. The module $\mathscr C_n(G)$ is equipped with the $G$--action
$$ g \mapsto \rho_g \: \mathscr C_{s(e)}(G) \to \mathscr C_{t(e)}(G) \qquad \rho_g(g_0, g_1, \ldots, g_n) := (gg_0, g_1, \ldots, g_n). $$
For $n \ge 1$ we define the boundary map: $\mathscr C_n(G) \to \mathscr C_{n-1}(G)$ by the formula
$$ \partial(g_0, \ldots, g_n) := \sum_{j=0}^{n-1} (-1)^j (g_0, \ldots, g_i\cdot g_{i+1}, g_n) + (-1)^n (g_0, \ldots, g_{n-1}). $$
We also have an augmentation
$$ \mathscr C_0(G) \to \R G \qquad g \mapsto t(g) \cdot 1, $$
where $\R G$ is the groupoid $\{\R_e\}_{e \in \obj(G)}$, where $G$ acts on $\R G$ by mapping $g$ to the map: $\id_\R \: G_{s(e)} \to G_{t(e)}$ (see \cite[Definition 3.2.4]{Blank}). 
Notice that we have equipped $\mathscr C_k(G)$ with a structure of bounded $G$--module, and that the boundary maps are $G$--linear. 

If $(G, A)$ is a pair of groupoids (i.e. if $A$ is a subgroupoid of $G$) we have an inclusion of complexes: $ \mathscr C_*(A) \hookrightarrow \mathscr C_*(G)$. The \textbf{relative bounded cohomology of $(G, A)$ with coefficients in $V$} is then given by the cocomplex
$$ \mathscr C_b^*(G, A; V) := \left\{f \in \hom_G^b(\mathscr C_*(G), V) \: f_e {}_{\Big| \mathscr C_*(A)_e} = 0 \, \forall e \in \obj(G) \right\} $$
and is denoted by $\mathscr H_b^*(G, A; V)$ (see \cite[Definition 3.5.1(iii), (iv)]{Blank}).

Let $(\Gamma, \Gamma' = \{\Gamma_i\}_{i \in I})$ be a group-pair. Let $\Gamma_I$ be the groupoid with $\obj(G) = I$, and $\morf(i, j) = G$, for all $i, j \in \obj(G)$. If $V$ is a bounded $\Gamma$--module, then $V_I$ denotes the bounded $\Gamma_I$--module $(V_i)_{i \in \obj(\Gamma_I)}$ with $\Gamma_I$--action given by $\rho_g (v) = gv$, where $v \in V_{s(g)}$ and $gv \in V_{t(g)}$. Let $\bigsqcup_{i \in I} \Gamma_i$ be the groupoid with $\obj(\bigsqcup_{i \in I} \Gamma_i) = I$ and $\morf(i, j) = G$ if $i = j$, and $\hom(i, j) = \emptyset$ otherwise (see \cite[Definitions 3.1.10, 3.5.11, Examples 3.1.3(iii)]{Blank}). The \textbf{relative bounded cohomology of the group-pair $(\Gamma, \Gamma')$ with coefficients in $V$} is defined to be the relative bounded cohomology of the corresponding groupoid-pair $(\Gamma_I, \bigsqcup_{i \in I} \Gamma_i)$, i.e.
$$ \mathscr H_b^*(\Gamma, \Gamma'; V) := \mathscr H_b^*(\Gamma_I, \bigsqcup_{i \in I} \Gamma_i; V) $$
(see \cite[Definition 3.5.12]{Blank}).

\begin{prop} \label{are naturally isometric} Let $(\Gamma, \Gamma')$ be a group-pair, and let $V$ be a bounded $\Gamma$--module. There is a natural isometric chain isomorphism $ \mathscr C_b^*(\Gamma, \Gamma'; V) \to \st_b^*(\Gamma, \Gamma'; V)$. 
\end{prop}

\begin{proof} We see an element in $\St_{(b)}^{\rel \, k}(\Gamma, \Gamma'; V) := \hom_{(b)}^\Gamma(\St_k^{\rel}, V)$ as a $\Gamma$--linear map $f \: \R (\Gamma \times I)^{k+1} \to V$ which is null on $St_k'$, i.e. on tuples $(x_0, \ldots, x_n)$ for which there exists $i \in I$ such that $x_j \in \Gamma \times \{i\}$ for all $0 \le j \le k$ and $x_j \in x_0\Gamma_i$ for every $1 \le j \le k$. If $i, j \in I$ and $g \in \Gamma$, we write $g^{i \to j}$ for the corresponding element in $\morf(i, j)$, and $g^i$ for the corresponding element in $\Gamma \times \left\{i\right\} \subset I\Gamma$. 

Fix $ \overline i \in I$ and consider the maps
$$ \varphi^k \: \mathscr C_b^k(\Gamma_I; V) \to St^k(\Gamma, \Gamma'; V) \qquad \psi^k \: St^k_\rel(\Gamma, \Gamma'; V) \to \mathscr C_b^k(\Gamma, \Gamma'; V) $$
defined as follows: if $f \in \mathscr C_b^k(\Gamma, \Gamma'; V)$ we put
$$ \varphi^k(f)\,\,(g_0^{i_0}, \ldots, g_n^{i_n}) := f(g_0^{i_0 \to \overline i}, (g_0^{-1} g_1)^{ i_1 \to i_0}, \ldots , (g_{n-1}^{-1} g_n)^{i_n \to i_{n-1}}). $$

If $h \in St^k_\rel(\Gamma, \Gamma'; V)$ we put
$$ \psi^k(h) \, (g_0^{i_0 \to \overline i}, g_1^{i_1 \to i_0}, \ldots, g_n^{i_n \to i_{n-1}}) = h(g_0^{i_0}, (g_0g_1)^{i_1}, \ldots (g_0\cdots g_n)^{i_n}). $$
The computations that show that $\varphi^*$ and $\psi^*$ are mutually inverse chain maps are similar to the ones that prove that the bar-resolution and the homogeneous bar-resolution are isomorphic. Indeed, they resemble dual versions of the ones in \cite[Chapter VI 13 (b)]{HiSta}. We simply note that those maps are well-defined, i.e. the restrictions of $\varphi^k(f)$ on $St_k'$ and of $\psi^k(h)$ on $\R\left(\bigsqcup_i \Gamma_i\right)^{n+1}$ are null. Indeed, if $g_0^{i}, \ldots, g_n^{i} \in (\Gamma_i \times \left\{i\right\})^{n+1}$, then $(g_0^{i \to \overline i}, (g_0^{-1} g_1)^{i \to i}, \ldots , (g_n^{-1} g_n)^{i \to i})$ is an $(n+1)$--tuple of elements in a $\Gamma$--translate of $\Gamma_i \subset \morf(i, i)$, and therefore $f$ is null on it. Conversely, if $(g_0^{i \to i}, g_1^{i \to i}, \ldots, g_n^{i \to i})$ is a tuple of elements in $\Gamma_i \subset \morf(i, i)$, then $(g_0^{i}, (g_0g_1)^{i}, \ldots (g_0\cdots g_n)^{i}) \in (\Gamma_i \times \left\{i\right\})^{n+1}$, hence $h$ is null over it.
\end{proof}

\bibliographystyle{amsalpha}
\bibliography{secondoArticoloRevisedVersion}

\end{document}